\documentclass{article}
 
\usepackage[margin = 2.5cm]{geometry}
\usepackage{amsmath}
\usepackage{amsfonts}
\usepackage{amssymb}
\usepackage{graphicx}
\usepackage{tabularx}
\usepackage{caption}
\usepackage{multirow}
\usepackage{mathrsfs}
\usepackage{lipsum}
\usepackage{epstopdf}
\usepackage{algorithmic}
\usepackage{amsopn}
\usepackage{adjustbox}
\usepackage{comment}
\usepackage{lineno}
\usepackage{url}
\usepackage{color}
\usepackage{graphicx}
\usepackage[hidelinks]{hyperref}
\usepackage{xcolor}
\usepackage{tabularx}
\usepackage{caption}   
\usepackage{todonotes}
\usepackage{adjustbox}
\usepackage{fancyhdr}
\usepackage{algorithm2e}
\usepackage{multirow}
\usepackage{booktabs}
\usepackage{amsthm}
\usepackage[normalem]{ulem}
\RestyleAlgo{ruled}
\newtheorem{remark}{Remark}
\newtheorem{lemma}{Lemma}
\newtheorem{lemmas}{Lemma}[section]
\def\linearmaps{\mathcal{L}}
\def\invertible{\mathcal{GL}}
\def\uparam{u^{\mu}}
\def\Aparam{A^{\mu}}
\def\uparambar{u^{\bar{\mu}}}

\def\Aparamfix{A^{\bar{\mu}}}
\def\Aparamj{A^{\mu_j}}
\def\fparam{b^{\mu}}

\def\fparambar{b^{\bar{\mu}}}
\def\zparambar{z^{\bar{\mu}}}
\def\identity{\operatorname{Id}}
\def\preconditioner{\mathscr{P}}
\def\risk{\mathfrak{R}}
\def\probability{\varrho}
\def\hypospace{\mathcal{H}}
\def\argmin{\operatorname*{arg\,min}}
\def\abilinear{a}
\def\bbilinear{b}
\def\forcingBold{\bb_h^\mu}
\def\tildeforcingBold{\tilde{\bb}_h^\mu}
\def\bb{\mathbf{b}}
\def\ff{\mathbf{f}}
\def\uu{\mathbf{u}}
\def\vv{\mathbf{v}}
\def\zz{\mathbf{z}}
\def\xx{\mathbf{x}}
\newcolumntype{Y}{>{\centering\arraybackslash}X}

\newcommand{\nicola}[1]{\textcolor{black}{#1}}
\newcommand{\unitsphere}{\mathbb{S}^{N_h-1}}

\begin{document}

\title{{\color{black} Numerical Solution of Mixed-Dimensional PDEs \\Using a Neural Preconditioner}}

\author{Nunzio Dimola$^1$, Nicola Rares Franco$^1$, Paolo Zunino$^{1,}$\footnote{Corresponding author\\\emph{Email addresses}: \textsf{nunzio.dimola@polimi.it}, \textsf{nicolarares.franco@polimi.it}, \textsf{paolo.zunino@polimi.it}}}

\date{\small $^1$MOX, Department of Mathematics, Politecnico di Milano, P.zza Leonardo da Vinci 32, Milan, 20133, Italy}

\maketitle

\begin{abstract}
Mixed-dimensional partial differential equations (PDEs) are characterized by coupled operators defined on domains of varying dimensions and pose significant computational challenges due to their inherent ill-conditioning. Moreover, the computational workload increases considerably when attempting to accurately capture the behavior of the system under significant variations or uncertainties in the low-dimensional structures such as fractures, fibers, or vascular networks, due to the inevitable necessity of running multiple simulations. In this work, we present a novel preconditioning strategy that leverages neural networks and unsupervised operator learning to design an efficient preconditioner specifically tailored to a class of 3D-1D mixed-dimensional PDEs. The proposed approach is capable of generalizing to varying shapes of the 1D manifold without retraining, making it robust to changes in the 1D graph topology. 
Moreover, thanks to convolutional neural networks, the neural preconditioner can adapt over a range of increasing mesh resolutions of the discrete problem, enabling us to train it on low resolution problems and deploy it on higher resolutions.
Numerical experiments validate the effectiveness of the preconditioner in accelerating convergence in iterative solvers, demonstrating its appeal and limitations over traditional methods. This study lays the groundwork for applying neural network-based preconditioning techniques to a broader range of coupled multi-physics systems.
\end{abstract}








\section{Introduction}
Research on mixed-dimensional partial differential equations (PDEs) is advancing to address challenges in complex multiscale applications in fields such as geophysics, biomechanics, and neuroscience. These disciplines often involve thin structures with high aspect ratio, including structural mechanics and geomechanics \cite{Meier2018124,Lespagnol2024,Fumagalli2013454}, microcirculation and perfusion modeling \cite{Vidotto20191076,Possenti20213356}, axons in neural networks \cite{Buccino2021}, and environmental flows \cite{Hagmeyer2022}. Modeling these interactions with traditional 3D meshes would require massive computational resources, making direct simulations impractical. Mixed-dimensional PDEs, which enable lower-dimensional representations to be embedded within higher-dimensional domains, present a promising alternative.
However, the computational effort associated with these models rises significantly
when trying to capture the behavior of the system for varying configurations or interactions within the low-dimensional structures, the latter representing, e.g., fracture networks in geological formations, fiber distributions in structural materials, or vascular and neural networks in biological systems. Similarly, multiple simulations are needed when the structure of the low-dimensional problem is uncertain or unknown. In all such cases, despite the simplifications provided by mixed-dimensional PDEs, the overall computational demand remains high due to the cumulative cost of repeated individual simulations.
This growing need underscores the importance of developing efficient solvers and preconditioners capable of handling the repeated solutions essential to representing the statistical variability of intricate networks in real-world applications.

Today, the discretization of mixed-dimensional PDEs is established, providing computationally efficient schemes suitable for applications with intricate networks \cite{10.1007/s10231-020-01013-1,10.1051/m2an/2019027,Kuchta2021558}, with rigorous error estimates ensuring accuracy. The well-established literature on block and operator preconditioning techniques, see, e.g., \cite{https://doi.org/10.1002/nla.716}, has been instrumental in the development of computational solvers for such problems. \cite{Kuchta2016B962,KMM2} have paved the way for the efficient resolution of discrete mixed-dimensional problems. Along this direction, we also mention the work of \cite{budisa2024algebraic} on algebraic multigrid methods for metric-perturbed coupled problems, which has significantly advanced the understanding of block preconditioning techniques in mixed-dimensional PDEs. Parallel to this, other researchers have also explored approximate algorithms that can enhance scalability for large systems: see, e.g. \cite{FIRMBACH2024117256}. 

The goal of this study is to develop a novel preconditioning strategy that addresses the twofold challenges inherent in mixed-dimensional PDEs: the ill-conditioned behavior of these problems and the need for adaptability across diverse geometric configurations of the low-dimensional subproblems. We aim to create a preconditioner that not only mitigates the ill-conditioning of the system, but also maintains this improved conditioning across a variety of configurations for 1D or other low-dimensional structures. To accomplish this, the preconditioner is designed as a nonlinear operator rather than a conventional matrix, enabling it to process the entire solution manifold effectively. Our method takes advantage of the latest advances in the approximation capabilities of neural networks and learning approaches related to machine learning \cite{KOVACHKI2024419,BOULLE202483,Lu2021}. These advancements have sparked active research in the creation of novel solvers and preconditioners tailored for different types of problem \cite{Kutyniok,doi:10.1137/24M162861X,grimm2023learningsolutionoperatortwodimensional,MARGENBERG2022110983,Azulay2023S127}, this being a non-exhaustive list of examples supporting the potential of neural network-based preconditioners as a viable alternative for solving large, sparse linear systems across various scientific and engineering domains.

Building upon these studies, we develop a nonlinear preconditioner designed to efficiently manage the numerical challenges associated with the complex structure of mixed-dimensional equations. By incorporating a shape descriptor into the learning framework, the preconditioner can generalize across different configurations of the 1D graph or low-dimensional structures without the need for retraining. 
In order to demonstrate the efficacy of the proposed approach, we focus our attention on a mixed-dimensional 3D-1D model problem, which serves as a representative template for a broader class of coupled systems with similar characteristics. The results show that the learned preconditioner significantly accelerates the convergence of the iterative Generalized Minimum Residual (GMRES) solver, even as the complexity of the problem increases due to changes in the 1D graph's topology or the size of the 3D domain. Unlike traditional preconditioning techniques, which typically require problem-specific tuning or reconfiguration when the underlying system changes, our method maintains consistent performance across diverse problem settings {\color{black}and mesh resolutions}. Moreover, the proposed preconditioning approach can be easily generalized, making it a versatile and robust tool for a wide array of mixed-dimensional problems beyond the specific examples addressed in this study.

 
This manuscript is organized as follows. First, we introduce the mathematical formulation of the 3D-1D mixed-dimensional PDE and highlight the challenges associated with solving such systems. Then, we present the proposed preconditioner learning framework, discussing both the theoretical foundations and the practical implementation of the nonlinear preconditioner. Next, we apply the preconditioner to the 3D-1D problem, illustrating its construction and providing a detailed analysis of its efficiency and scalability. Finally, we validate the approach through extensive numerical experiments, comparing its performance against conventional preconditioning methods, and conclude with a discussion of the implications and potential future directions for extending this work.

\section{A mixed-dimensional problem}
\label{sec:model}
We consider here a mixed-dimensional elliptic problem as a simple template for a wider class of coupled problems with similar characteristics. 
It will be our reference for comparisons, analyses, and considerations on preconditioner performances. Originally, we define the problem in a three-dimensional (3D) slender domain embedded in an external one, where we set partial differential equations coupled by suitable interface conditions. The mixed-dimensional model is obtained by a \textit{dimensional reduction} strategy that represents the slender interior domain as a one-dimensional (1D) manifold, namely a metric graph. 
More precisely, we address a coupled problem defined by an exterior domain $\Omega\subset\mathbb{R}^3$ and an interior domain $\Sigma\subset\Omega$. We assume $\Sigma$ to be a generalized cylinder with a centerline $\Lambda$, the latter being a 1D domain with arc length parameter $s\in(0,S)$. We assume that the cylinder has a constant radius $\epsilon>0$, so that if $\lambda:(0,S)\to\Lambda$ parameterizes the centerline, then
$\Sigma=\{\lambda(s)+\varepsilon\;:\;s\in(0,S),\|\varepsilon\|\le\epsilon,\;\varepsilon\perp\tau(s)\}$
with $\tau(s)$ the tangent vector at $\Lambda$.
We assume that the transverse diameter of $\Sigma$, which is equal to $2\epsilon$, is small compared to the diameter of $\Omega$. 
The dimensionality reduction process consists of replacing $\Sigma$ with $\Lambda$ while keeping some information about its original 3D structure and its interaction with the external domain $\Omega$. Specifically, let $\Gamma:=\partial\Sigma$ be the interface of the coupled problem in the full 3D representation and let $\mathcal{T}_\Lambda$ be the restriction operator from $\Omega$ to $\Lambda$. In \cite{Kuchta2021558} it is defined as the composition of the trace operator on \(\Gamma\) combined with a projection operator form \(\Gamma\) to \(\Lambda\) based on cross-sectional averages.
Then, the continuous 3D-1D coupled problem can be formally formulated as follows,
\begin{equation}
\begin{cases}
-\nabla \cdot \left(k_\Omega\nabla u_\Omega\right)+\sigma_\Omega u_{\Omega}+2\pi\epsilon\left(\mathcal{T}_{\Lambda} u_\Omega-u_{\Lambda}\right)\delta_{\Lambda}=0, & \text{on}\; \Omega,\\
-d_s \left(k_{\Lambda}\partial_s u_{\Lambda}\right)+2\pi\epsilon\left(u_{\Lambda}-\mathcal{T}_{\Lambda} u_\Omega \right)=0, & \text{on}\; \Lambda, \\ 
-\nabla u_\Omega \cdot \mathbf{n}=0, & \text{on}\; \partial \Omega, \\
-d_s u_{\Lambda} \cdot \mathbf{n}=0, & \text{on}\; \partial \Lambda \setminus \partial\Lambda_D, \\
u_{\Lambda}=1, & \text{on} \; \partial \Lambda_D.
\end{cases}
\label{eq:problem}
\end{equation} 
The functions $u_\Omega, u_\Lambda$ are the coupled unknowns on the 3D and the 1D domains, respectively, and $k_\Omega, \sigma_\Omega , k_\Lambda$  and $2\pi\epsilon$ represent respectively the coefficients of the elliptic operators on the 3D domain, the 1D one and the coupling parameter. Here, $\partial\Lambda_{D}\subset\partial\Lambda$ is a given subset of the 1D domain holding a Dirichlet boundary condition. 
An analogous formulation holds when the low-dimensional structure is given by the union of multiple substructures $\Sigma=\cup_i\Sigma_i$ and $\Lambda=\cup_i\Lambda_i$, each being a generalized cylinder. In this case, $\Lambda$ becomes a 1D graph embedded within a 3D domain. Figure~\ref{fig:3d1d_sol} shows an example of a solution to this mixed-dimensional problem (1D solution on the left, 3D solution in the center).
\begin{figure}[htb!]
    \centering
    \includegraphics[height=0.27\linewidth]{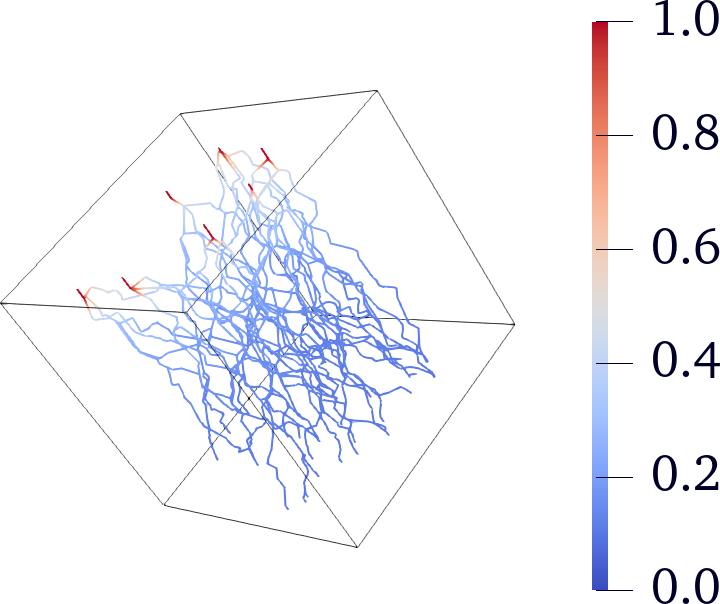}
    \includegraphics[height=0.27\linewidth]{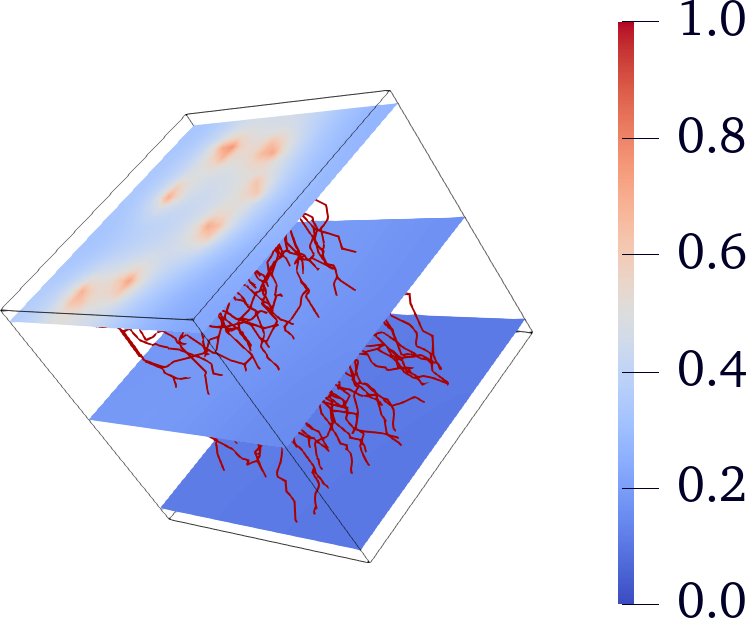}
    \includegraphics[height=0.27\linewidth]{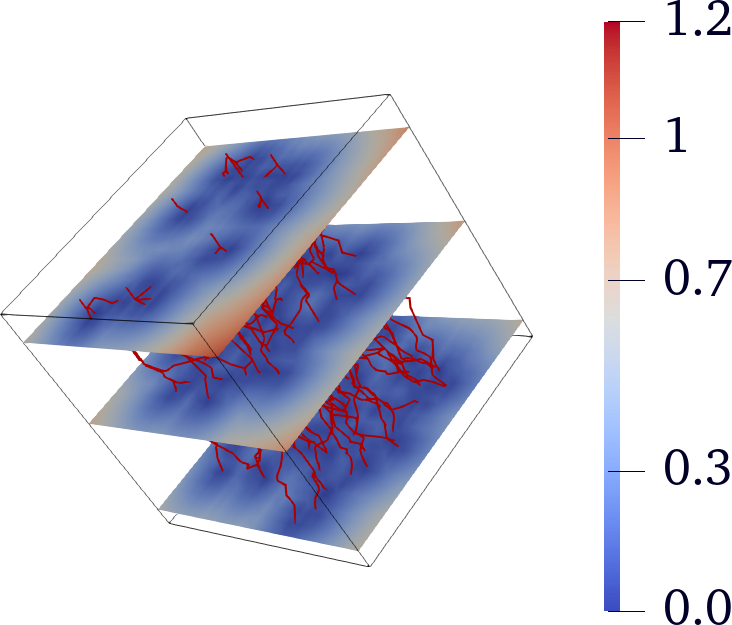}
    \caption{\small \small Solution of a 3D-1D coupled problem for a fixed graph geometry $\Lambda$. \textbf{Left:} plot of 1D solution $u_\Lambda$. \textbf{Center:} slices of 3D solution $u_\Omega$. \textbf{Right:} slices plot of the distance function $d(\Lambda)$ associated to the graph.}
    \label{fig:3d1d_sol}
\end{figure}

\subsection{Parametrization and discretization of the problem}
\label{subsec:discretization}
As we mentioned, we are interested in solving multiple instances of the above problem for different configurations of the low-dimensional structure, including, for instance, changes in
the topology, arc-length or density of the 1D graph $\Lambda$.
To formalize this fact, it is useful to define a suitable \textit{parameter space}, $\mathcal{P}$, that collects all configurations of interest. To this end, we first note that each 1D graph $\Lambda\subset\Omega$ is uniquely identified by its \textit{3D distance function}, $d(\Lambda):\Omega\to\mathbb{R}$,
which maps each point $x\in\Omega$ to its distance from the closest point in $\Lambda$: see the right panel of Figure~\ref{fig:3d1d_sol} for a visual representation.
This allows us to define the parameter space as a subset of the space of continuous functions on $\Omega$, that is, $\mathcal{P}:=\{d(\Lambda)\}_{\Lambda}\subset\mathcal{C}(\Omega).$ This representation has several advantages. First, it can be easily transferred to the discrete setting by introducing a suitable discretization of the exterior domain $\Omega$ using, for example, Finite Elements (FE). Furthermore, it automatically enriches the parameter space by equipping it with a metrizable topology (the one induced by the supremum norm): this allows us to discuss, for instance, about the continuity of the 3D solution with respect to the 1D graph. Finally, it will be useful when discussing our neural preconditioning approach in Section~\ref{sec:Realization}.

We now come to the discretization of \eqref{eq:problem}. To this end, we first address its variational formulation.
Let $V := H^1(\Omega) \times H^1_{\partial\Lambda_D}(\Lambda)$ be the mixed-dimensional space where PDE solutions will be sought. Here, $H^1_{\partial\Lambda_D}(\Lambda)=\{w\in H^{1}(\Lambda):w_{|\partial\Lambda_{d}}\equiv1\}$ accounts for Dirichlet boundary conditions, so that ---up to translations--- $V$ is isomorphic to a Hilbert space. In the parametric setting, re-writing Eq. \eqref{eq:problem} in weak form corresponds to saying that,  for every $\mu\in\mathcal{P}$ we aim to find $u^{\mu} \in V$ such that
\begin{equation}\label{eq:problem_weak}
\abilinear^{\mu}(u^{\mu}, v) = F^{\mu}(v), \quad \forall v \in V,
\end{equation}
with $\abilinear^{\mu}: V \times V \mapsto \mathbb{R}$ and $F^{\mu}: V \mapsto \mathbb{R}$ parameter-dependent operators. 
Specifically,
\begin{equation*}
    \abilinear^{\mu}(u, v) = \abilinear_\Omega(u_\Omega,v_\Omega) + \abilinear_\Lambda^\mu(u_\Lambda,u_\Lambda) + \bbilinear_\Lambda^\mu(\mathcal{T}_\Lambda u_\Omega - u_\Lambda, \mathcal{T}_\Lambda v_\Omega-v_\Lambda),
\end{equation*}
where, being $\mathcal{D}(s)$ the cross section of $\Sigma$ at arc-lengh $s$, and having set \nicola{$$(u,v)_{\Omega}:=\int_{\Omega}u(x)v(x)dx,\quad(u,v)_{\Lambda, |\mathcal{D}|}:=\int_\Lambda u(s) v(s) |\mathcal{D}(s)| ds,\quad (u,v)_{\Lambda, |\partial\mathcal{D}|}:=\int_\Lambda u(s) v(s) |\partial\mathcal{D}(s)| ds,$$} we have,
\begin{align}
    \label{eq:aomega}
    \abilinear_{\Omega}(u,v) &= (k_\Omega \nabla u, \nabla v)_{\Omega} + (\sigma_\Omega u, v)_{\Omega}, \\
    \label{eq:alambda}
    \abilinear_{\Lambda}^\mu(u,v) &= (k_\Lambda d_{\mathbf{s}} u, d_{\mathbf{s}} v)_{\Lambda, |\mathcal{D}|},
    \quad
    \bbilinear_{\Lambda}^\mu(u,v) = 2\pi\epsilon (u, v)_{\Lambda, |\partial \mathcal{D}|}.
\end{align}
We note that the adopted notation is slightly redundant because the parametric dependence of $\abilinear^{\mu}$ is only due to the fact that $\abilinear_\Lambda^\mu$ and $\bbilinear_\Lambda^\mu$ have support in $\Lambda$.


In order to discretize \eqref{eq:problem_weak}, we consider its Galerkin projection onto a (broken) FE space $V_h = V_h^\Omega \times V_h^\Lambda$ of dimension $N_h=N_h^\Omega + N_h^\Lambda$, being $N_h^\Omega = \textrm{dim}(V_h^\Omega)$ and $N_h^\Lambda = \textrm{dim}(V_h^\Lambda)$, which is suitably chosen depending on the characteristics of the problem at hand. 
Assuming for simplicity a fully conformal approximation, given $V_h \subset V$ with \(V_h^\Omega \subset H^1(\Omega) , \ V_h^\Lambda \subset H^1_{\partial\Lambda_D}(\Lambda)\), we aim to find $u_h^\mu \in V_h$ such that
\begin{equation}\label{eq:problem_weak_h}
\abilinear^{\mu}(u^\mu_h, v_h) = F^{\mu}(v_h) \qquad \forall v_h \in V_h.
\end{equation}
From a discrete point of view, the problem \eqref{eq:problem_weak_h} is equivalent to a (large) system of algebraic equations \(A^{\mu}_{h}\uu_{h}^\mu = \mathbf{F}_h^\mu\). Specifically, we have
\begin{equation}
\label{eq:mainsys}
\overbrace{\left(
\underbrace{\left[\begin{matrix}
  K_{h,00} &  0\\[10pt] 
  0  & K_{h,11}^\mu \\
\end{matrix}\right]}_{\displaystyle K_h^\mu}
+2\pi\epsilon
\underbrace{\left[\begin{matrix}
  M_{h,00}^\mu & M_{h,01}^\mu \\[10pt] 
  M_{h,10}^\mu & M_{h,11}^\mu
\end{matrix}\right]}_{\displaystyle M_h^\mu}
\right)}^{\displaystyle A^{\mu}_{h}}
\overbrace{\left[\begin{matrix}
 \uu_{h,0}^\mu  \\[10pt]
 \uu_{h,1}^\mu   \\
\end{matrix}
\right]}^{\displaystyle\uu_{h}^\mu}
= \overbrace{\left[\begin{matrix}
 0  \\[10pt] 
 \ff_{h,1}^\mu   \\
\end{matrix} \right]}^{\displaystyle\mathbf{F}_h^\mu},
\end{equation}
where $\uu_{h}^\mu \in \mathbb{R}^{N_h}$ is the vector of degrees of freedom of the FE approximation;
$K_{h,00}$ and $K_{h,11}^\mu$ represent the discretized bilinear forms \(\abilinear_\Omega\) and  \(\abilinear_\Lambda^\mu\); 
the block matrix $M_h^\mu$ models the coupling enforced by the operator \(\bbilinear_\Lambda^\mu(\mathcal{T}_\Lambda u_\Omega - u_\Lambda, \mathcal{T}_\Lambda v_\Omega-v_\Lambda)\), and the forcing term $\ff^\mu_{h,1}$ takes into account the non-homogeneous Dirichlet condition on the 1D graph.

As we discretize the differential problem, we also take the opportunity to discretize the parameter space, which we do by projecting ---or interpolating--- all 3D distance functions onto $V_h^\Omega$. This allows us to represent each parametric configuration $\mu\in\mathcal{P}$ through a finite-dimensional vector $\mathbf{d}_{h}^\mu \in \mathbb{R}^{N_h^\Omega}$, obtained by listing the nodal values of $\mu=d(\Lambda)$ at the degrees of freedom. In this way, the discrete/finite-dimensional parameter space becomes \(\mathcal{P}_h \equiv \{ \mathbf{d}_{h}^\mu\} \subset \mathbb{R}^{N_h^\Omega} \).

\begin{remark}
    \nicola{As apparent from equations \eqref{eq:aomega}-\eqref{eq:alambda}, we notice that problem \eqref{eq:problem_weak} is symmetric positive definite. However, we will not exploit this property when defining the computational solver. This choice is motivated by two main reasons. The first one is that \eqref{eq:problem_weak} is just a particular case of a general family of operators in which the bilinear form $a_\Omega(\cdot,\cdot)$ may also describe transport phenomena that violate symmetry. The second reason is that, as it will be clearer later on, learned techniques, such as our neural preconditioner,
    can hardly preserve symmetry, unless this constraint is explicitly built into the network architecture---a strategy that currently lacks a reliable and efficient implementation.}
\end{remark}

\begin{remark}
    In what follows, we shall make the assumption that the 3D domain has been discretized using a uniform mesh, consisting either of tetrahedrons or bricks. This requirement will be crucial in Section~\ref{sec:Realization} when implementing our neural network-based preconditioner as it will allow us to exploit the relationship between tensor-like structures and convolutional neural networks. Nevertheless, we emphasize that in the context of mixed-dimensional problems, where the primary challenge is representing a complex low-dimensional structure embedded in 3D, this restriction is acceptable. In fact, this requirement does not constrain the low-dimensional domain, which can still exhibit arbitrary geometric complexity. 
\end{remark}

\subsection{A preconditioning strategy for the coupled problem}
\label{subsec:preconditioning}
As we mentioned, solving \eqref{eq:mainsys} can be computationally intensive due to poor conditioning of the linear system. Here, we address this fact by using a block-preconditioning approach, specifically tailored for the block-structured matrix arising from the discretization of the coupled 3D-1D problem. 
The goal is to employ a right-preconditioner expressed as:
\begin{equation}
\begin{aligned}
&\hspace{25pt} A_h^{\mu}Q_h^\mu \mathbf{z}_h^\mu=\mathbf{F}_h^\mu, \quad Q_h^\mu \mathbf{z}_h^\mu=\mathbf{u}_h^\mu,
\end{aligned}
\label{eq:right-prec full problem}
\end{equation}
where $Q_h^\mu$ is the inverse of the block upper triangular part of $A_h^\mu$.
Employing this right preconditioner to accelerate the convergence of an iterative method, see e.g. \cite{saad2003iterative}, involves conditioning the residual from $[{\zz}_{h,0}^\mu, {\zz}_{h,1}^\mu]^\top$ to $[\tilde{\zz}_{h,0}^\mu, \tilde{\zz}_{h,1}^\mu]^\top$ by solving the following linear system through a two-step back-substitution procedure,
\begin{equation}\label{eq:twosteps}
\left[
\begin{matrix}
(K_{h,00}+2\pi\epsilon M_{h,00}^\mu) &2\pi\epsilon M_{h,01}^\mu \\
0 &(K_{h,11}^\mu+2\pi\epsilon M_{h,11}^\mu)
\end{matrix}
\right]
\left[
\begin{matrix}
    \tilde{\zz}^\mu_{h,0}\\
    \tilde{\zz}^\mu_{h,1}
\end{matrix}
\right]
=
\left[
\begin{matrix}
    \zz_{h,0}^\mu\\
    \zz_{h,1}^\mu
\end{matrix}
\right].
\end{equation}
In general, the first step, which is to solve for $\tilde{\mathbf{z}}_{h,1}^{\mu}$, does not pose significant challenges. This is because the dimension of the discrete 1D problem is typically smaller than the one of the 3D problem, $N_h^\Lambda\ll N_h^\Omega$, making it more manageable to address using traditional techniques, such as direct solvers. In contrast, the second step, which is to solve for $\tilde{\mathbf{z}}_{h,0}^{\mu}$, can be computationally demanding and can severely hinder the applicability of the approach when considering many-query scenarios, where the linear system is to be solved for multiple instances of the model parameters. This step, in fact, is the one where the main challenges inherent to mixed-dimensional problems become apparent. Indeed, as noted in \cite{budisa2024algebraic}, the latter involves elliptic operators altered by a low-order semidefinite component, referred to as the \emph{metric term}. This term, whose impact is modulated by a scalable intensity parameter $\epsilon$, can disrupt the structure and conditioning of the elliptic component, thereby complicating the numerical approximation of the coupled system.
For these reasons, our objective is to develop a nonlinear preconditioner, $\mathscr{P}$,
for the mixed-dimensional system defined on $\Omega$,
\begin{equation}
(K_{h,00}+2\pi\epsilon M_{h,00}^\mu)\tilde{\zz}_{0,h}^\mu = \zz_{h,0}^\mu-2\pi\epsilon M^\mu_{01}\tilde{\zz}_{1,h}^\mu.
\label{eq:decoupled_3d_problem}
\end{equation} 
This nonlinear preconditioner will take two arguments: one that is conformal to the 3D solution (as any classical preconditioner would do), and one that refers to the problem parameters, allowing the preconditioner to adjust depending on the scenario of interest. We shall write $\mathscr{P}=\mathscr{P}(\mathbf{v},\mathbf{d})$ in the discrete setting and, occasionally, $\mathscr{P}=\mathscr{P}(v,\mu)$ when considering the more abstract continuous formulation. Using this notation, and having set
\begin{equation}
\label{eq:matrixC}
C_h^\mu :=(K_{h,00}+2\pi\epsilon M_{h,00}^\mu), \quad
\vv_h^\mu := \tilde{\zz}_{0,h}^\mu, \quad
\forcingBold :=\zz_{h,0}^\mu-2\pi\epsilon M^\mu_{01}\tilde{\zz}_{1,h}^\mu,
\end{equation}
the right-preconditioned family of problems reads $C_h^\mu \mathscr{P}(\mathbf{v}_h^\mu, \mathbf{d}_h^\mu)=\forcingBold.$
Considering the specific structure of $\mathbf{F}_h^\mu$ in \eqref{eq:mainsys}, we focus on making the preconditioner $\mathscr{P}$ efficient for a limited subset of acceptable right-hand sides within $\mathbb{R}^{N_h^\Omega}$. In particular, we are interested in solving the system for the forcing terms that arise from the solution of the 1D problem, namely all $\mathcal{B}=\{\forcingBold\}_{\mu\in\mathcal{P}}$ such that
\begin{equation}\label{eq:setB}
\mathcal{B}:=\left\{-2\pi\epsilon M_{01, h}^\mu \xx\;:\mu \in \mathcal{P}\;\text{and}\; \xx\in\mathbb{R}^{N_h^\Omega}\;\text{with}\; (K_{11}^\mu+2\pi\epsilon M_{11, h}^\mu)\xx = \mathbf{f}_h^\mu \right \}.
\end{equation}
As we anticipated, to guarantee generalization across various configurations of the 1D problem, it is essential to develop a suitable nonlinear preconditioner, $\mathscr{P}$. This involves carefully choosing an appropriate hypothesis space and designing a representative training set, as detailed in the following section.

\section{A learning approach for preconditioning parametrized systems}
\label{Sec: OpLearning}
We shall now discuss the problem of learning preconditioners for parametrized systems.
Although this task is inherently related to finite-dimensional spaces, we find it convenient to explore and present the idea within an infinite-dimensional context.
This shift to a function space setting, as presented in this section, provides a deeper insight into the mathematical structure of the problem. For these reasons, we start this section by fixing some notation. Given two Banach spaces $(V_1, \|\cdot\|_{V_1})$ and $(V_2, \|\cdot\|_{V_2})$ we denote $\linearmaps(V_1,V_2)$ the space of bounded linear operators from $V_1$ to $V_2$, equipped with the operator norm
$$\|A\|_{\linearmaps(V_1,V_2)}:=\sup_{v\in V_1\setminus\{0\}}\displaystyle\frac{\|Av\|_{V_2}}{\|v\|_{V_1}}.$$
It is also useful to introduce the \textit{General Linear Group}
\begin{equation*}
\invertible(V_1, V_2):=\{A\in\linearmaps(V_1, V_2)\;:\;\exists B\in \linearmaps(V_2, V_1),\;BA=\identity_{V_1}\;\text{and}\;AB=\identity_{V_2}\},
\end{equation*}
which consists of all linear operators from $V_{1}\to V_{2}$ that are both invertible and continuous. Here, $\identity_{V_i}$ denotes the identity operator of $V_i$ on itself.
Given a Banach space $(V,\|\cdot\|_V)$, we also denote by $V'$ its topological dual, that is, $V':=\linearmaps(V,\mathbb{R})$.
Our purpose is to address the problem of operator preconditioning for parametrized linear problems in arbitrary (complete) normed spaces, i.e. to find a continuous preconditioner operator $\mathscr{P}=\mathscr{P}(v,\mu)$ for a parametrized problem of the form 
\begin{equation}
\label{eq:problema-astratto}
\text{find}\; \uparam\in V:\quad \Aparam \uparam = \fparam, \quad \forall \mu \in \mathcal{P},
\end{equation}
given a Banach state space $(V,\|\cdot\|_V)$ and a compact metric space $\mathcal{P}$ serving as parameter space. Here, $\Aparam\in\invertible(V,V')$ and $\fparam\in V'$ are parameter-dependent operators and problem data. In general, we think of the latter as a linear system arising from the discretization of a given PDE, such as \eqref{eq:problem_weak_h} or \eqref{eq:mainsys}. Everything, in fact, can be traced back to our discussion in Section~\ref{sec:model}: we refer to Remark~\ref{remark:continuous-discrete} for additional insight on the matter. Furthermore, we anticipate that the upcoming Section~\ref{sec:Realization} will also contain further clarifications on how these concepts translate when put into practice.

We organize the remainder of this Section into distinct parts. Initially, we explore how operator preconditioning is applied to \eqref{eq:problema-astratto}, emphasizing the imperative for a nonlinear preconditioner (Section~\ref{subsec:nonlinear}). Subsequently, we outline the structure of the learning algorithm, presenting a suitable unsupervised training strategy (Section~\ref{subsec:learning}). 

\begin{remark}
    \label{remark:continuous-discrete}
    The advantage of working within an abstract environment is that the idea can be easily transferred from the continuous formulation to the discrete one and vice versa. For instance, following our notation in Section~\ref{subsec:preconditioning}, if we resort to the discrete setting, then $A^\mu\cong C_h^\mu$ 
    is a matrix, while $u^\mu\cong \mathbf{v}_h^\mu$ and $\fparam \cong \tildeforcingBold$ are finite-dimensional vectors; cf. Eq. \eqref{eq:decoupled_3d_problem}. Conversely, if we lift the idea to the continuous level, then $A^\mu$ is the linear operator acting as $u\mapsto a_\Omega(u,\cdot)+b_\Lambda^\mu(\mathcal{T}_\Lambda u,\cdot)$, while $u^\mu$ and $\fparam$ are functions defined over $\Omega$: cf. Section~\ref{subsec:discretization}. 
\end{remark}

\subsection{The need for a nonlinear and matrix-free preconditioner}
\label{subsec:nonlinear}
Classical approaches to operator preconditioning focus on finding a linear preconditioner for a given operator of interest $A\in\linearmaps(V',V)$. In principle, one could apply this idea to parametric problems by addressing each parametric scenario separately.
Indeed, for a generic, but \textit{fixed}, parameter $\bar{\mu} \in \mathcal{P}$
one could consider the corresponding operator $\Aparamfix\equiv A$, with $A:V\to V'$, and then look for a suitable (right) preconditioner
that is, any linear operator $P\in\linearmaps(V',V)$ for which \(\kappa(AP)\ll \kappa(A)\), being $\kappa$ the operator \textit{condition number} $\kappa(B):=\|B\|_{\mathcal{L}{(V_1, V_2)}} \|B^{-1}\|_{\mathcal{L}(V_1, V_2)}$ for $ B\in\invertible(V_1,V_2)$. 
Then, once a suitable preconditioner is available, the original problem $A \uparambar = \fparambar$ can be replaced with a new one, $AP \zparambar = \fparambar$, which is easier to solve, so that $\uparambar = P\zparambar$. In general, we note that up to isometries, the "optimal" preconditioner is $P=A^{-1}$, regardless of the right hand side $\fparambar \in V'$. However, this approach becomes computationally unfeasible if the procedure is repeated multiple times for varying $\mu\in\mathcal{P}$.
A first alternative could be to look for a surrogate $\tilde{\mathscr{P}}:\mathcal{P}\to\linearmaps(V',V)$ that, for each $\mu\in\mathcal{P}$, returns a suitable preconditioner for $A^\mu u^\mu = \fparam$. However, we do not find this solution to be very practical as that would entail casting the learning problem within an extremely high-dimensional space. Notice, in fact, that in the discrete setting one has $\dim\linearmaps(V',V)=(N_h^\Omega)^2$. Furthermore, even if we restrict our attention to sparse preconditioners, this is still very challenging as the sparsity pattern would likely depend on the structure of the underlying 1D problem, i.e. on $\mu\in\mathcal{P}$.
A more practical approach can be to rely on a matrix-free formulation. There, one is not directly interested in the preconditioner by itself, but rather in the way it acts when applied to an input vector. Mathematically speaking, this change of perspective corresponds to looking for a map, potentially \textit{nonlinear}, of the form $\mathscr{P}:V'\times\mathcal{P}\to V.$ This idea is further motivated by the following Lemma.
\begin{lemma}
\label{lemma:nonlinear}
Let $\mathcal{P}\ni\mu\mapsto \Aparam\in\linearmaps(V,V')$ be continuous. For any $\mu\in\mathcal{P}$ and any $v\in V'$ let $x^{\mu,v}$ be the solution of $A^\mu x^{\mu,v}=v.$
Then, there exists a continuous (nonlinear) operator $\preconditioner: V' \times \mathcal{P} \to V$ such that 
$$\preconditioner(v, \mu) = x^{\mu,v} \quad \forall (v, \mu) \in V' \times \mathcal{P}.$$
In particular, $A^\mu\mathscr{P}(v,\mu)=v$ for all $(v,\mu)\in V'\times\mathcal{P}.$
\end{lemma}

\begin{proof}
Let $\mathscr{A}:\mathcal{P}\to\invertible(V,V')$ be the map $\mathscr{A}(\mu) = \Aparam$. Let $\mathscr{I}:\invertible(V,V')\to \invertible(V',V)$ be the inversion map, $B\mapsto B^{-1}$. It is well known that $\mathscr{I}$ is continuous --- for a rigorous proof, we refer the interested reader to the Appendix, Lemma~\ref{lemma:inversion}.
Let $\preconditioner: V' \times \mathcal{P} \to V$ be defined as
$\preconditioner(v, \mu) := \left[(\mathscr{I}\circ \mathscr{A})(\mu)\right](v).$
Then, $\preconditioner$ is continuous and $\mathscr{P}(v,\mu)=(A^\mu)^{-1}v=x^{\mu,v}$, as claimed.
\end{proof}
As seen in Lemma~\ref{lemma:nonlinear}, in order to adapt to different parametric scenarios, the preconditioner $\mathscr{P}$ must be nonlinear in its arguments (see also Remark~\ref{remark:nonlinear}). 
In light of this, hereon, when speaking of a nonlinear preconditioner for a parametric problem such as \eqref{eq:problema-astratto}, we intend any continuous map $\mathscr{P}:V'\times\mathcal{P}\to V$ for which $\Aparam\preconditioner (v, \mu) \approx v$, or equivalently,
$(\identity -  \Aparam\preconditioner(\cdot, \mu))v\approx 0$.
Ideally, this should hold for all $\mu\in \mathcal{P}$ and all right-hand sides of interest $v\in \mathcal{K}^\mu\subset V'$, with the latter subset potentially depending on the parameter instance $\mu$. For example, the simplest choice could be $\mathcal{K}^\mu=\{\fparam\}$, so that $\mathscr{P}(\fparam,\mu)\approx u^\mu$. Setting $\mathcal{K}^\mu=V'$ is also possible but not strictly necessary: we recall, in fact, that our ultimate goal is to solve \eqref{eq:problema-astratto}.
We shall return to this in Section~\ref{subsec:learning}, as we address the details of our learning strategy, and in Section~\ref{subsec:Kmu}, where we discuss the actual implementation of the approach and the definition of $\mathcal{K}^\mu$.
\begin{remark}
    \label{remark:nonlinear}
    In principle, one could restrict oneself to continuous maps $\mathscr{P}:V'\times\mathcal{P}\to V$ that are linear in their first argument, $\mathscr{P}=\mathscr{P}(v,\mu)$, as it happens for the "ideal" preconditioner in Lemma~\ref{lemma:nonlinear}. However, allowing for more general maps gives additional freedom in designing the preconditioner, potentially simplifying the learning process. For example, assume that $\mathcal{P}$ is compact and that the map $\mathscr{F}:\mu\to\fparam$ is both continuous and injective. Then, $\mathscr{F}^{-1}:\mathscr{F}(\mathcal{P})\to\mathcal{P}$ exists, is continuous and can be extended to $V'\supset\mathscr{F}(\mathcal{P})$ using, e.g., Dugundji's extension of Tieze's Theorem \cite{dugundji1951extension}. With this setup, consider the simple scenario $\mathcal{K}^{\mu}=\{b^\mu\}$ and let $\mathscr{P}$ be the map in Lemma 1. 
    Define $\hat{\mathscr{P}}:V'\times\mathcal{P}\to V$ as $\hat{\mathscr{P}}(v,\mu):=\mathscr{P}(v,\mathscr{F}^{-1}(v))$. Then, $\hat{\preconditioner}\neq\preconditioner$ is nonlinear in $v$, does not depend explicitly on $\mu$ (only implicitly), but still $A^\mu\hat{\preconditioner}(v,\mu)=v$ for all $\mu\in\mathcal{P}$ and all $v\in \mathcal{K}^\mu\subset V'$.
\end{remark}

\subsection{Learning strategy}
\label{subsec:learning}
Following our previous discussion, our purpose here is to learn a (nonlinear) preconditioner for a parametrized problem
given as in Eq. \eqref{eq:problema-astratto}.
In practice, this involves fixing a suitable \textit{hypothesis space} $\hypospace$ and finding the optimal preconditioner $\mathscr{P}\in\mathcal{H}$ by minimizing a suitable functional known as \textit{loss function}.
In light of Lemma~\ref{lemma:nonlinear}, we
let $\hypospace\subset\mathcal{C}(V'\times \mathcal{P},V)=\{\preconditioner:V'\times \mathcal{P}\to V\;\textnormal{continuous}\}$ be a subset of the space of continuous functions from $V'\times\mathcal{P}$ to $V$. For instance, if we resort to deep learning techniques, as we shall do in the remaining of the paper, the latter can be thought of as the space of all possible neural network models associated to a given architecture of choice. In this case, we'll refer to $\preconditioner$ as to a \textit{neural preconditioner}.
Concerning the loss function, instead, different approaches become available. In principle, one could define the nonlinear preconditioner $\preconditioner$ by requiring either $\preconditioner(v, \mu)\approx x^{\mu,v}$ or $(\identity -  \Aparam\preconditioner(\cdot, \mu))v\approx0$. In fact, we notice that if the original problem admits a unique solution for each $\mu\in \mathcal{P}$, then every $\preconditioner$ that fulfills $\preconditioner(\fparam, \mu)=\uparam$ yields $\Aparam\preconditioner(\fparam, \mu) = \fparam$ and vice versa. In this sense,
the two approaches may seem equivalent.
However, our hypothesis space will typically not contain the global minimizer in Lemma~\ref{lemma:nonlinear}, reason for which minimizing the \textit{error }$\|\preconditioner(v, \mu)-x^{\mu, v}\|$ or the \textit{residual} $\|(\identity -  \Aparam\preconditioner(\cdot, \mu))v\|$ can provide completely different results. Furthermore, the two approaches differ significantly in their implementation: the first one is \textit{supervised}, as it requires actually solving \eqref{eq:problema-astratto} in order to compute $x^{\mu,v}$ for multiple $\mu\in \mathcal{P}$; in contrast, the second approach is fully \textit{unsupervised}. 

This work focuses exclusively on the second approach. This is because preliminary investigations by the authors found the supervised approach to be extremely inefficient, with high costs during the training phase, significant challenges in minimizing the loss function, and unsatisfactory results.
In contrast, the unsupervised approach is more practical, as it does not require sampling from the solution manifold $\mathcal{S}=\{\uparam\}_{\mu\in \mathcal{P}}$. 
All of this considered, a first definition of the learning problem can be
\begin{equation}
\label{eq:learning0}
\preconditioner_{*}=\argmin_{\preconditioner\in\hypospace}\;\int_{\mathcal{P}}\|(\identity-\Aparam\preconditioner(\cdot, \mu))\fparam\|_{V'}\probability(d\mu),
\end{equation}
for a given probability measure $\probability$ defined over the parameter space $\mathcal{P}$. The downside of this approach is that for each $\mu\in \mathcal{P}$, we focus solely on how well $\preconditioner$ operates on $\fparam$.  
As anticipated in Section~\ref{subsec:nonlinear}, a better approach can be to introduce a subset $\mathcal{K}^\mu\subset V'$ of possible right-hand-sides (depending in general on $\mu$), and reformulate \eqref{eq:learning0} as
\begin{equation}
\label{eq:learning1}
\preconditioner_{*}=\argmin_{\preconditioner\in\hypospace}\;\int_{\mathcal{P}}\int_{\mathcal{K}^\mu}\|(\identity-\Aparam\preconditioner(\cdot, \mu))v\|_{V'}\probability_\mu(dv)\probability(d\mu),
\end{equation}
where $\probability_\mu$ is a suitable probability measure supported over  $\mathcal{K}^\mu\subset V'$. Clearly, \eqref{eq:learning0} is a special case of \eqref{eq:learning1}, since it can be obtained letting $\mathcal{K}^\mu=\{\fparam\}$ be a singleton and $\probability_{\mu}=\delta_{\fparam}$ be a Dirac delta measure. The question of defining $\mathcal{K}^\mu$ and the corresponding probability measure is generally intriguing, yet quite complex. We will defer this discussion to Section~\ref{subsec:Kmu}. We anticipate that $\mathcal{K}^\mu$ should contain all the components that are detrimental to the well-posedness of $A^\mu$ and, in general, to the convergence rate of the algebraic solver considered. 

From a practical point of view, the ideal minimization problem in Eq. \eqref{eq:learning1} can be addressed through \textit{empirical risk minimization} \cite{KOVACHKI2024419,BOULLE202483}. Specifically, let $\mu_{1},\dots,\mu_{N_{\mathcal{P}}}$ be an independent random sample identically distributed (i.i.d.) with $\mu_{j}\sim\probability$. For each $j=1,\dots,N_{\mathcal{P}}$ let $\mathcal{K}^{\mu_j}:=\{v_{1,j},\dots,v_{N_{\mathcal{K}^{\mu_j}}, j}\}\subset \mathcal{K}^{\mu_j},$
be an i.i.d. random sample with $v_{i,j}\sim\probability_{\mu_j}$. Consider the empirical measures
\begin{equation*}
    \probability_{N_{\mathcal{P}}} = \frac{1}{N_{\mathcal{P}}} \sum_{j=1}^{N_{\mathcal{P}}} \delta_{\mu_j},\quad \probability_{N_{\mathcal{K}^{\mu_j}}} = \frac{1}{N_{\mathcal{K}^{\mu_j}}} \sum_{v\in \mathcal{K}_{\mu_j}}\delta_{v},\quad
    \probability_{N^{\mathcal{K}^{\mu}}} = \sum_{j=1}^{N_{\mathcal{P}}}\mathbf{1}_{\{\mu_j\}}(\mu) \probability_{N_{\mathcal{K}^{\mu_j}}},
\end{equation*}
where $\delta_{x}$ is the Dirac measure centered at $x$. Following classical procedures, we define the empirical risk as
\begin{equation}
\begin{aligned}  
    \label{eq:risk}
    \risk(\preconditioner)&:=\int_{\mathcal{P}}\int_{\mathcal{K}^\mu}\|(\identity-\Aparam\preconditioner(\cdot, \mu))v\|_{V'}\probability_{N_{\mathcal{K}^\mu}}(dv)\probability_{N_\mathcal{P}}(d\mu)=\\ 
    &=\frac{1}{N_{\mathcal{P}}}\sum_{j=1}^{N}\frac{1}{N_{\mathcal{K}^{\mu_j}}}\sum_{v\in \mathcal{K}^{\mu_j}}\|(\identity-\Aparamj\preconditioner(\cdot, \mu_j))v\|_{V'}.
\end{aligned}
   \end{equation}
Then, the \textit{training phase} consists in solving the following minimization problem
\begin{equation}
    \label{eq:trained_preconditioner}
    \tilde{\preconditioner}_{*}:=\argmin_{\preconditioner\in\hypospace}\;\risk(\preconditioner).
\end{equation}
Note that while this step might be computationally demanding, it has to be performed only once. After training, the preconditioner can be readily applied to any problem in the parametric class, without further optimizations or assembling stages. In particular, if one designs the hypothesis space to consist of neural network models, this can provide significant speed ups, as the computational cost associated with online evaluations of $\tilde{\preconditioner}_*$ becomes nearly negligible (on the order of milliseconds).

\section{Realization of the learning approach}
\label{sec:Realization}
In this section, we return to the finite-dimensional setting to make the preconditioner learning strategy specific to the discretization of the parametrized mixed-dimensional problem \eqref{eq:decoupled_3d_problem}-\eqref{eq:matrixC}. 
As we address here the 3D problem arising from the elimination of the 1D equations from the coupled system, with a slight abuse of notation, we write $N_h$ instead of $N_h^\Omega$. Once a suitable preconditioner $\mathscr{P}:\mathbb{R}^{N_h} \times \mathbb{R}^{N_h}\to \mathbb{R}^{N_h}$ has been identified (we recall that the finite-dimensional parameter space is $\mathcal{P}_h \subset \mathbb{R}^{N_h}$), it can be used to accelerate the solution of the discrete problem. In general, this can be achieved in multiple ways. Here, we propose the integration of the nonlinear preconditioner within a GMRES solver,{\color{black} motivated by the fact that we cannot guarantee the symmetry of the neural preconditioner}. In addition, considering the nonlinear nature of the preconditioner $\mathscr{P}$, we shall adopt the Flexible GMRES (FGMRES) algorithm \cite{saad1993flexible}, as the latter can easily accommodate variations of the preconditioning operator with the input vector.

\subsection{The training set $\mathcal{K}^\mu$}
\label{subsec:Kmu}
The selection of the training set \( \mathcal{K}^\mu \) is a crucial factor in the success of the operator learning approach for preconditioning. In fact, the training set must match the underlying behavior of the operator to ensure that the learned preconditioner is effective in different instances of the parameter space \( \mathcal{P} \). 
To start, we notice that inside the FGMRES algorithm, the preconditioner will act only on vectors from the unit sphere $\unitsphere:=\{\mathbf{v}\in\mathbb{R}^{N_h}\;:\;\|\mathbf{v}\|=1\}$. Thus, we let $\mathcal{K}^\mu\subset\mathbb{S}^{N_h-1}.$ Next, based on the findings presented in \cite{budisa2024algebraic}, we propose that the training set should be guided by the kernel structure of the operator. In fact, the coupling between the 3D and 1D domains introduces significant challenges in the construction of efficient solvers, particularly because the kernel of the operator \( M_{h,00}^\mu \) contains high-frequency components that must be taken into account by the computational solvers. In the context of the algebraic multigrid (AMG) approach addressed in \cite{budisa2024algebraic}, it is shown that an essential property for the robustness of the AMG solver is the following kernel decomposition condition:
\begin{equation*}
    \mathrm{ker}(M_{h,00}^\mu) = \mathrm{ker}(M_{h,00}^\mu) \cap V_c + \sum_{j=1}^{J} \mathrm{ker}(M_{h,00}^\mu) \cap V_j\,,
\end{equation*}
where in the case of AMG algorithms \(V_c\) refers to the coarse subspace that captures the low-frequency components of the solution and \(V_j\) are the fine subspaces that represent localized corrections in specific regions or aggregates. In the context of a multiresolution method such as AMG, this condition suggests that the numerical solver should operate on all components of the kernel of the coupling operator. The richer the kernel subspace, possibly including low- and high-frequency components, the harder it is to develop good solvers.
This issue is closely related to the concept of \textit{spectral bias} in neural network training, which states that lower-frequency components are learned more easily and earlier during training, while higher-frequency components are more challenging to learn \cite{rahaman2019spectralbiasneuralnetworks,zhiqin2024frequency}. For the approximation of mixed-dimensional problems, this spectral bias has direct implications. Specifically, the high-frequency components in the kernel of the coupling operator \( M_{h,00}^\mu \) may not be adequately captured if the training process is not designed to address this inherent difficulty.
To mitigate this challenge, we include in \( \mathcal{K}^\mu \) samples that span both the low-frequency and high-frequency components of the kernel of \( C^{\mu}_h \), defined in \eqref{eq:matrixC}. In the context of operator learning, this insight is reflected in the definition of the measure $\varrho_\mu(dv)=\varrho_{N_{\mathcal{K}^\mu}}$. This considered, we propose the following training set structure:
\begin{equation}
    \mathcal{K}^\mu = \bigcup_{j=1}^{N_{\mathcal{P}}} \mathcal{K}^{\mu_j}, \quad  \mathcal{K}^{\mu_j}:= \left\{\frac{b_h^{\mu_j}}{||b_h^{\mu_j}||} \right\} \cup \mathcal{D}^{\mu_j}, \quad b_h^{\mu_j} \in \mathcal{B}; 
\end{equation}
where $\mathcal{D}^{\mu_j}$ represents a suitable data-augmentation set, given the parameter $\mu_j$ (see Fig.~\ref{fig:TrainSet_scheme}). In this way, the operator $\mathscr{P}(\cdot, \mu)$ can be trained to act as a preconditioner on a suitable subset of the unit sphere with a specified frequency content.
In what follows, we examine two distinct strategies for data augmentation, while the next section will present numerical experiments that support this methodology and demonstrate that incorporating both right-hand side vectors and kernel components into \( \mathcal{K}^\mu \) improves the convergence rates in FGMRES with the application of the learned preconditioner.
\begin{figure}
     \centering
     \includegraphics[width=0.7\linewidth]{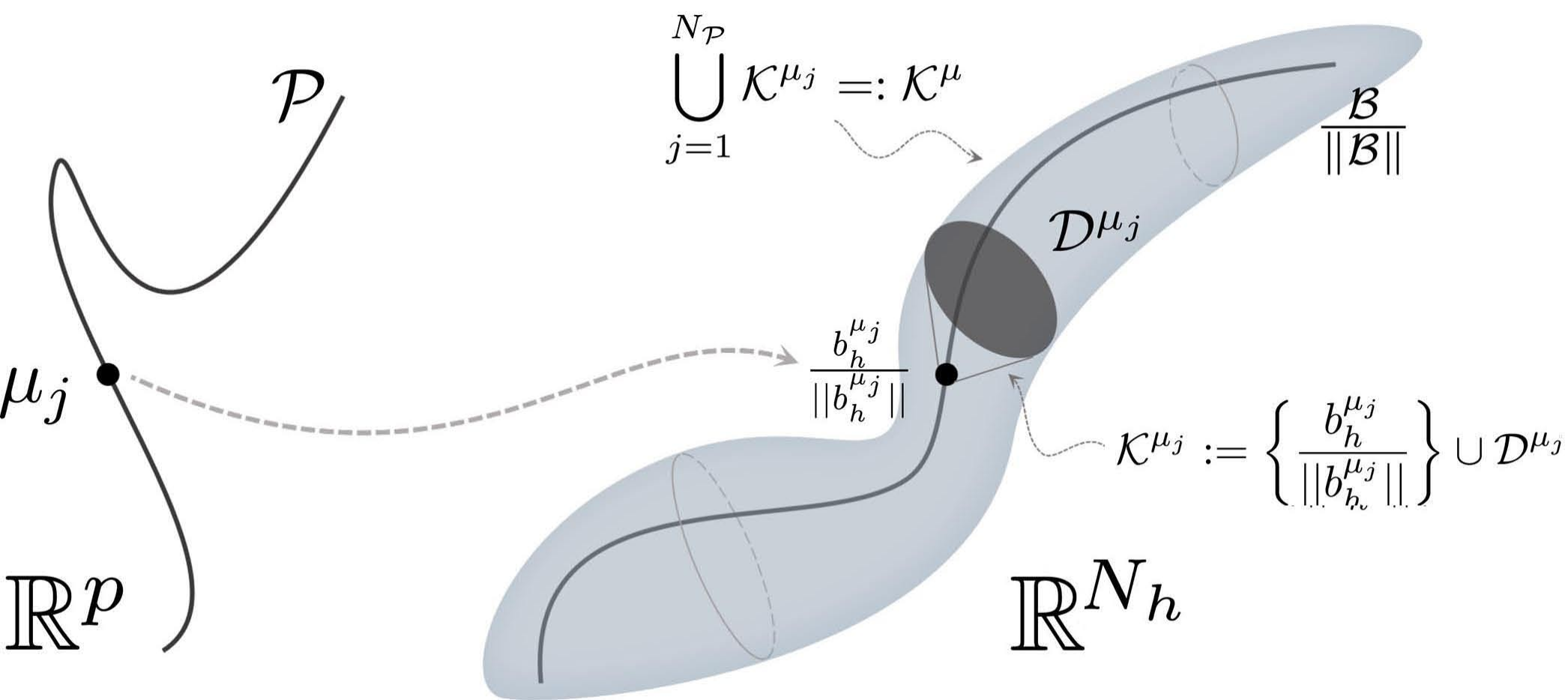}
     \caption{\small \small Pictorial summary for the construction of the augmented training set $\mathcal{K}^\mu$. The data augmentation set $\mathcal{D}^{\mu_j}$ contains vectors in the unit sphere with the desired frequency content. 
     Here, $\mathcal{B}/||\mathcal{B}||:=\left\{b/||b||: b \in  \mathcal{B}\right\} \subset \unitsphere$ is the set of normalized right-hand side vectors.}
     \label{fig:TrainSet_scheme}
\end{figure}

\subsubsection{Augmented Training Set by Krylov Subspaces}
\label{SubSec:Kry_train}
Our initial data-augmentation strategy involves 
incorporating high-frequency elements derived from the construction of Krylov subspaces related to the matrix. As a starting point, the Krylov subspace associated with a general operator \( O \in \mathcal{L}(X,X) \) and a vector \( x \in X \) is defined as:
\[\mathcal{K}_q(O, x) := \text{span}(x, Ox, O^2 x, \ldots, O^{q-1}x)\]
where \( q \) is the order of the Krylov subspace. This subspace contains progressively higher-order components of \( x \) under the action of the operator \( O \), capturing information about the behavior of the system across multiple scales. By including vectors from these Krylov subspaces in our training set, we effectively augment the representation of both low- and high-frequency components in the solution space.

{\color{black}Specifically, for a generic parameter $\mu_j$, let \( O = C_h^{\mu_j} \) be the discrete operator associated with the mixed-dimensional PDEs described in previous sections. We construct the data-augmentation set $\mathcal{D}^{\mu_j}$ as
$$
\mathcal{D}^{\mu_j} := \left\{ q_{p^\prime}^{\mu_j} \ldots q_{(p+p^\prime)}^{\mu_j} \right\},$$
where $0 \leq p^\prime$ is a suitable shift index and $q_{i}^{\mu_j},\,i=0,\ldots,p^\prime+p$
is a suitable orthonormal basis of $\mathcal{K}_{(p+p^\prime)}(C_h^{\mu_j}, b_h^{\mu_j})$ with $q_0^{\mu_j} \equiv b_h^{\mu_j}/\|b_h^{\mu_j}\|$.
These additional basis functions aim to capture a wider spectrum of the operator's behavior, thus enhancing the preconditioner’s ability to address different scales and frequencies present in the solution.
In this context, it is convenient to keep \( p \) and \( p^\prime \)small and of the same magnitude (i.e. $p\simeq p^\prime$)}.  This lies in the balance between representing the essential high-frequency characteristics of the operator \( C_h^\mu \) and maintaining computational efficiency. The choice of small \( p,p^\prime \) allows us to efficiently construct the augmented training set without significantly increasing the computational burden during the training phase. As will be shown later (\ref{fig:spectrum}), even with a small value of \( p+p^\prime \), the Krylov subspace captures important high-order effects that improve the robustness of the preconditioner.\\

\subsubsection{Augmented Training Set by Random Vectors}
\label{SubSec:hf_train}
A second approach involves adding to the physics-based training subset $\mathcal{B}$ a new set of completely unrelated random functions. Specifically, we propose augmenting the training set by generating random unity vectors $\mathbf{r}_h \in \unitsphere \subset \mathbb{R}^{N_h}$, uniformly sampled from the unit hypersphere. In practice, each $\mathbf{r}_h$ can be generated by drawing and normalizing a random vector $\mathbf{v} \sim \mathcal{N}(\mathbf{0}, \mathbf{I}_{N_h})$ with standard multivariate normal distribution, that is, $ \mathbf{r}_h = \mathbf{v} / \|\mathbf{v}\| $.
%
%
{\color{black}Such vectors are designed to introduce high-frequency content that is independent of the physics of the problem, thereby complementing the functions based on physics \( \mathbf{b}_h^\mu  \in \mathcal{B}\).} In fact, given a finite element discretization based on a FEM space \( V_h(\Omega) \), where \( \Omega \) represents the three-dimensional computational domain and \( h \) denotes the characteristic element size, the vectors \( \mathbf{r}_h \) can be interpreted as the highest frequency modes that can be resolved by the mesh. The frequency associated with these random vectors is proportional to \( |\Omega|/h \), which corresponds to the smallest possible wavelength that can be captured in the discretized domain. 
Mathematically, we can define the data augmentation set as \( \mathcal{D}^{\mu_j} := \left\{ r_{h,1}, r_{h,2}, \ldots \right\}\). Note that these vectors are independent of $\mu$ but a different instance of $\mathcal{D}$ is defined for each $\mu_j$.
We expect this strategy to improve robustness with respect to high-frequency vectors that are inside $\ker(M_{h,00}^\mu)$.

\subsection{Hypothesis space for $\mathscr{P}$}
\label{sec:unet}
At this stage, it remains to define the hypothesis space $\mathcal{H}$, over which the empirical risk minimization will be performed \eqref{eq:trained_preconditioner}. Following a deep-learning approach, the hypothesis space is specified once a particular \textit{neural architecture} $\mathcal{N}_H$ is selected. This architecture is characterized by a composition of linear and nonlinear functions, which in turn are defined by a set of \textit{hyper-parameters} $H$. The resulting hypothesis space $\mathcal{H}$ can thus be expressed as $\mathcal{H} = \{\mathcal{N}_H(\cdot \;; \theta)\}_{\theta \in \mathbb{R}^t}$ where $\theta$ is the vector of trainable parameters associated with $\mathcal{N}_H$, over which the optimization is performed.
Drawing inspiration from cutting-edge deep learning techniques in image processing, as proposed in \cite{Azulay2023S127}, which employ varying resolutions to effectively manage the diverse scales of images through pooling and upscaling operations, the neural architecture chosen to represent the preconditioner will be a \textit{U-net}, see for example \cite{ronneberger2015u} for a landmark paper and \cite{williams2024unifiedframeworkunetdesign} for recent developments. This architecture is highly capable of addressing multi-scale issues as it effectively captures both global and local features across varying resolutions. This makes it a strong candidate function for representing solutions of mixed-dimensional problems, such as the ones represented in Figure~\ref{fig:3d1d_sol}. According to this choice, we define $\mathcal{N}_H(\cdot\;; \theta)\equiv\mathcal{U}_L(\cdot\;; \theta)$, where $\mathcal{U}_L$ represents a U-net architecture with $L$ levels. For the sake of simplicity, we will refer only to the hyper-parameter $L$ to indicate the depth of the U-net, although additional parameters such as bottleneck size, number of channels, and convolutional kernel sizes are also involved. Consequently, the desired nonlinear preconditioner operator can be represented as $\mathscr{P} \equiv \mathcal{U}_L(\cdot\; \theta^\star)$, where $\theta^\star$ denotes the set of optimized parameters.
The U-net architecture is structured as a combination of two components, an encoder function $\Phi$ and a decoder function $\Psi$:
\[\Phi := \Phi_{0} \circ \Phi_{1} \circ \ldots \circ \Phi_{L- 2} \circ \Phi_{L-1}: \mathbb{R}^{c_{in} \times n}  \rightarrow \mathbb{R}^{c_b \times m };\]
\[\Psi := \Psi_{L-1} \circ \Psi_{L-2} \circ \ldots \circ \Psi_{1} \circ \Psi_{0}: \mathbb{R}^{ c_b \times m} \rightarrow \mathbb{R}^{c_{out} \times n}.\]
Then, given an input tensor ${\bf X}=[{\bf X}_1|{\bf X}_2|...|{\bf X}_{c_{in}}]$, where ${\bf X}_i \in \mathbb{R}^{n_1 \times n_2 \times \ldots \times n_d}$, and a prescribed number of output channels $c_{out}$, a U-net with $j$ levels can be defined recursively as follows: 
\[\mathcal{U}_j({\bf X}) := \Psi_j \left([\mathcal{U}_{j-1} \circ \Phi_j | \Phi_j]\right)({\bf X}), \quad \; \mathcal{U}_0({\bf X}) := \Psi_0 \circ \Phi_0({\bf X}),\]
where the notation $[a | b]$ represents a \textit{channel stacking} operation that concatenates tensors $a$ and $b$ along the channel dimension. Specifically, if $a \in \mathbb{R}^{c_a  \times m}$ and $b \in \mathbb{R}^{c_b \times m }$, then we have $[a | b] \in \mathbb{R}^{(c_a+c_b)\times m}$.
Operation $[\mathcal{U}_{j-1} \circ \Phi_j ({\bf X})| \Phi_j({\bf X})]$ is called the $j$-th \textit{ skip connection}, which transmits information from the $j$-th encoding layer to the corresponding $j$-th decoding layer. Skip connections are critical in retaining and merging high-resolution features, ensuring that the decoder can effectively recover fine-scale information. In the literature, the term $c$ refers to the \textit{number of channels} at each layer. Specifically, $c_b$ is the number of channels in the bottleneck layer of the U-net, while $c_{in}$ and $c_{out}$ denote the number of input and output channels, respectively.
The convolutional layer operates locally by sliding a kernel across the input tensor, extracting features from small spatial neighborhoods. Given a \(d\)-dimensional input tensor \({\bf X} \in \mathbb{R}^{c_{in} \times n_1 \times n_2 \times \ldots \times n_d}\), the convolution operation is defined as:
\[
    {\bf Y}_i = \sum_{j=1}^{c_{in}} {\bf k}_{i,j} \ast {\bf X}_j \ \text{for} \; i = 1, 2, \ldots, c_{out}, 
    \ \textrm{with} \ 
    \left({\bf k}_{i,j} \ast {\bf X}_j\right)(\mathbf{p}) = \sum_{\mathbf{q} \in \mathbb{Z}^d} {\bf k}_{i,j}(\mathbf{q}) \, {\bf X}_j(\mathbf{p} - \mathbf{q}),
\]
\({\bf k}_{i,j} \in \mathbb{R}^{s_1 \times s_2 \times \ldots \times s_d}\) being the kernel for the \((i,j)\)-th input-output channel pair, \(\mathbf{p}\) indexing the output tensor, and \(\mathbf{q}\) spanning the kernel positions.
This operation enables convolutional layers to act as local feature extractors, capturing patterns from the input tensor. Input data, often provided as ordered vectors, must be reshaped into tensor format to preserve spatial correlations. Once trained, the kernel can operate on input tensors of arbitrary dimensions, provided the structure is compatible, such as tensor-product meshes. As a result, the U-Net architecture can handle inputs of any size, which may increase its application beyond the dimensions of the training data. However, reliance on structured grids restricts the application of convolutional layers to simple geometries. Future work will address this limitation by integrating mesh-informed neural networks, as proposed in \cite{FrancoMINN}, to generalize U-net architectures while retaining their feature extraction capabilities.
\begin{figure}[h!]
    \centering
    \includegraphics[width=0.9\linewidth]{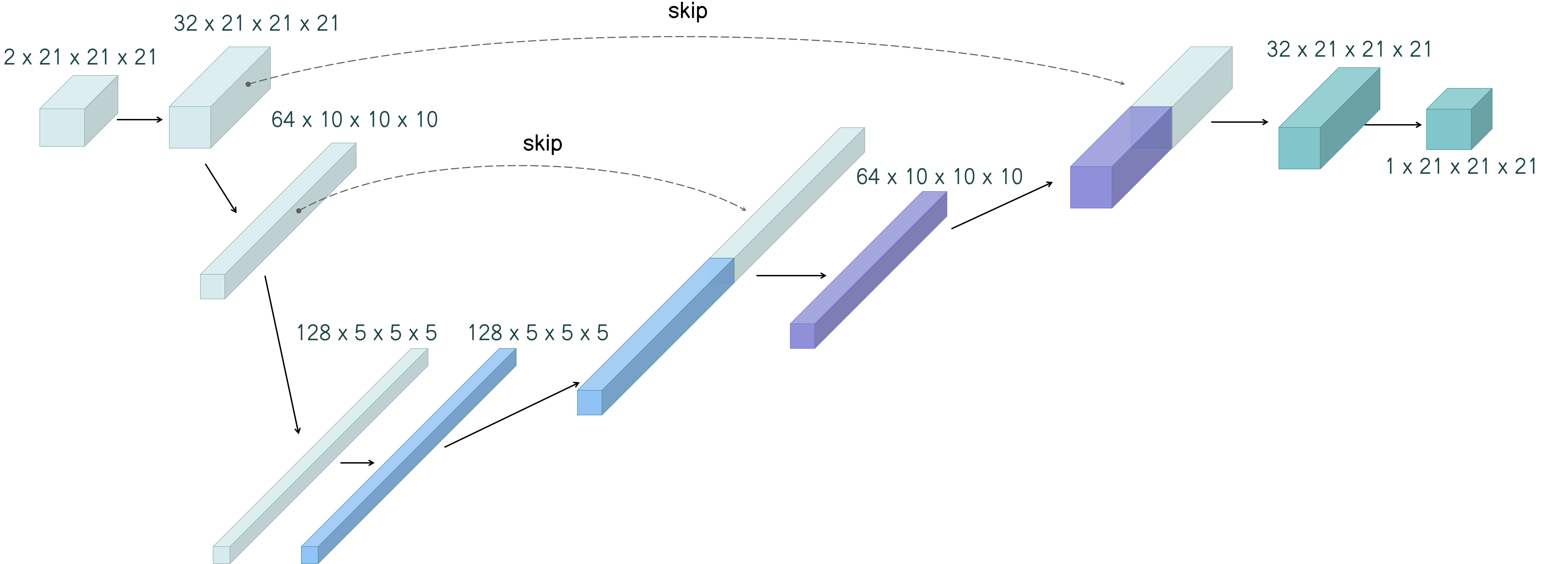}
    \caption{\small Schematic representation of the U-Net architecture $\mathcal{U}_3$. Tensor data are represented by blocks, where the number of channels corresponds to the depth of the blocks. The tensor shape, $c \times n_1 \times n_2 \times n_3$, is reported on the top of the blocks. Solid line arrows represent layer action, while dashed arrows represent channel-staking skip connection.}
    \label{fig:unet-arch}
\end{figure}

\begin{table}[htpb!]
    \centering
    \begin{adjustbox}{width=\textwidth}
    
    \begin{tabular}{lccccc}
    \hline
        \textbf{Layer type} & \textbf{input size} & \textbf{output size} & \textbf{input channels} & \textbf{output channels} & \textbf{\# parameters $\theta$} \\ \hline 
        \textbf{Conv.} & $21^3$ & $21^3$ & 2 & 32 & 13,840 ~ \\ 
        \textbf{Conv.+Max P.} & $21^3$ & $10^3$ & 32 & 64 & 55,296 ~ \\ 
        \textbf{Conv.+Max P.} & $10^3$ & $5^3$ & 64 & 128 & 221,184 ~ \\ 
        \textbf{Conv.} & $5^3$ & $5^3$ & 128 & 128 & 442,368 ~ \\ 
        \textbf{Trasp. conv. + Skip} & $5^3$ & $10^3$ & 128 & 64 & 286,784 ~ \\ 
        \textbf{Trasp. conv. + Skip} & $10^3$ & $21^3$ & 64 & 32 & 71,912 ~ \\
        \textbf{Conv.} & $21^3$ & $21^3$ & 32 & 1 & 545 ~ \\  
        \hline \hline
    \end{tabular}
\end{adjustbox}
\caption{\small \small Summary of the $\mathcal{U}_3$ architecture. Input data flows from the top row to the bottom row; see also the scheme in Fig.\ref{fig:unet-arch}. The total number of trainable parameters is $\theta \approx 10^6$.}\label{tab:unet}
\end{table}

\section{\color{black}Numerical solution of mixed-dimensional PDEs using the neural preconditioner}
We analyze how the neural preconditioner accelerates the convergence of iterative solvers, in the case of mixed-dimensional PDEs. We discuss the computational setup and analyze the preconditioner's impact on FGMRES convergence rate, mesh size scalability, and computational time. 
{\color{black}
In doing so, we shall take the opportunity to assess the importance of the several design choices entailed by the implementation of the neural preconditioner, such as, e.g., the data-augmentation strategy and the parametrization of the 1D geometry.
For the sake of this preliminary exploration, in this Section we shall restrict ourselves to numerical solution of problem \eqref{eq:matrixC}, that is, to the 3D block of the mixed-dimensional system, thus ignoring the 3D-1D coupling. The latter will be addressed right after, in Section~\ref{sec:coupled}.

We anticipate that the main message emerging from this numerical exploration is that the neural preconditioner proves competitive when compared to other state-of-the-art algorithms ---such as AMG or ILU---  thanks to a reduction in the computational cost of each preconditioning step. However, the iteration count and the scalability of the neural preconditioner with respect to the mesh size are not optimal. In fact, as shown later in Table~\ref{tab:iteration_and_time_FULL}, a trade-off between computational cost and algorithmic efficiency appears. However, our findings indicate that, in general, the neural preconditioner offers a competitive approach to solving mixed-dimensional PDEs, which motivates its use in Section~\ref{sec:coupled} when addressing the coupled problem.}

\subsection{General setup of the numerical tests}
The domain $\Omega$ is the unit cube and the problem is discretized using a structured mesh with $9261$ nodes ($21^3$ grid points). The physical constants are set as $k_\Omega = 10^{-3}$, $\sigma_\Omega=10^{-3}$ and $\epsilon=10^{-3}$.
\begin{figure}
    \centering
    \includegraphics[width=0.9\linewidth]{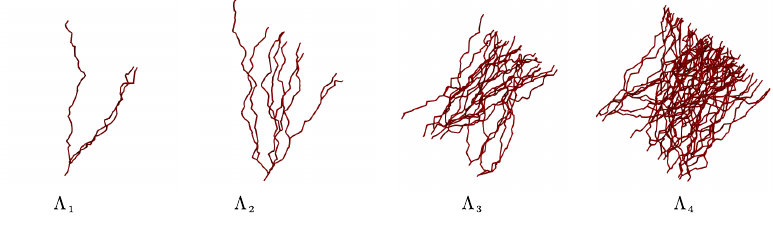}
    \caption{\small  Examples of the graphs $\Lambda$ considered in the numerical tests, with increasing geometrical complexity.}
    \label{fig:graphs}
\end{figure}
In addition, we performed tests using $k_\Omega=10^{-2},\, \sigma_\Omega=10^{-2}$, confirming the adaptability of the method. For the training set $\mathcal{K}^\mu$, different 1D graph geometries are considered (see Fig.~\ref{fig:graphs}), ranging from $O(1)$ to $O(10^2)$ number of branches. 

We considered a three-level U-net architecture, $\mathcal{U}_3$, with two input channels: one relative to the vector in the training set $\mathcal{K}^\mu$ and one carrying the parameter information $\mu$, that is, taking the discrete distance function $\mathbf{d}_h^\mu \in \mathbb{R}^{N_h}$ associated with the graph $\Lambda$. 
More technical details can be found in Table~\ref{tab:unet} and Fig.~\ref{fig:unet-arch}. 
The actual training set $\tilde{\mathcal{B}} \approx \mathcal{B}$ is obtained by solving approximately the system in \eqref{eq:setB}, with a relative tolerance of $10^{-4}$, thus ensuring computational efficiency in training set generation. The data-augmentations sets $\mathcal{D}^{\mu_j}$ vary depending on the scenario analyzed, resulting in three different U-Net preconditioners:
\begin{equation*}
    \begin{aligned}
        \mathcal{D}^{\mu_j} = \emptyset &\quad \Rightarrow \quad \mathcal{U}_3^\emptyset (\cdot\;; \theta), \\
        \mathcal{D}^{\mu_j} = \{q_5^{\mu_j}, \ldots, q_8^{\mu_j}\} &\quad \Rightarrow \quad \mathcal{U}_3^{kry}(\cdot\;; \theta), 
        \\
        \mathcal{D}^{\mu_j} = \{r_{h,1}, \ldots, r_{h,4}\} &\quad \Rightarrow \quad \mathcal{U}_3^{hf} (\cdot\;; \theta).
    \end{aligned}
\end{equation*}

Here, $\{q_5^{\mu_j}, \ldots, q_8^{\mu_j}\}$ are orthonormal bases of the Krylov subspace $\mathcal{K}_8(C_h^\mu, b_h^\mu)$, and $\{r_{h,1}, \ldots, r_{h,4}\}$ are high-frequency vectors sampled from the unit sphere $\unitsphere$.
The networks $\mathcal{U}_3$ are trained through the minimization of the empirical risk:
\begin{equation*}
    \mathcal{U}_3^\star =  \mathcal{U}_3(\theta^\star) 
    \ \textnormal{with} \
    \theta^\star = \arg\min_\theta =\frac{1}{N_{\mathcal{P}}}\sum_{j=1}^{N_{\mathcal{P}}}\frac{1}{N_{\mathcal{K}_{\mu_j}}}\sum_{\mathbf{v}\in \mathcal{K}_{\mu_j}} \left\| \mathbf{v} - \alpha^{-1} C_h^{\mu_j} \mathcal{U}_3(\mathbf{v}, \mathbf{d}_h^{\mu_j}; \theta) \right\|^2,
\end{equation*}
where $\alpha:= \left(\frac{1}{N_{\mathcal{P}}}\sum_{j=1}^{N_{\mathcal{P}}}\frac{1}{\sqrt{N_h}}\|diag(C_h^{\mu_j})\|_2\right)$ is a fixed normalizing constant representing the mean matrices diagonal root mean square,  which helps stabilize the training process. For comparison purposes, an equal number of graph samples and training epochs is used for $\mathcal{U}_3^\emptyset$, $\mathcal{U}_3^{kry}$, and $\mathcal{U}_3^{hf}$. More data on the training process are reported in Table~\ref{table:training_data}.

\begin{table}[h!]
    \centering
    \begin{tabular}{l c c | l c c }
    \hline
        \textbf{Training} & \textbf{1D graphs}& \textbf{Tot. samples}& \textbf{Validation}           & \textbf{1D graphs}& \textbf{Tot. samples}\\ \hline
        Samples& 480              & 2400                 &Samples& 120              & 600  \\ \hline 
        Batch Size                     &  & 5                  & Epochs &                       & 250                \\ 
        Learn.  rate& & \(10^{-3}\)                & Learn. strategy        & & decaying   \\ \hline
    \end{tabular}
    \caption{\small \small Details on the training dataset and parameters. Due to data augmentation strategies, the total number of training samples is five times the one of the 1D graphs.}
    \label{table:training_data}
\end{table}

Figure~\ref{fig:traintest} displays the convergence history of training and validation for the U-Net architecture when the physical constants $k_\Omega = 10^{-3}$, $\sigma_\Omega=10^{-3}$ are considered. 
The graphs illustrate that the training and validation error for $\mathcal{U}_3^\emptyset$, $\mathcal{U}_3^{kry}$ and $\mathcal{U}_3^{hf}$ possesses similar trends but quite different values; nevertheless, as will be clarified later on, the training and validation error alone are not discriminant for the performance of the U-Net preconditioner. What is fundamental is the data set over which the given error is obtained; e.g., even if $\mathcal{U}_3^\emptyset$ and $\mathcal{U}_3^{kry}$ possess similar error behavior in the validation set, their performance as a preconditioner is quite different (cf. Table~\ref{tab:precond_comparison}).
\begin{figure}[h]
    \centering
    \includegraphics[width=0.85\linewidth]{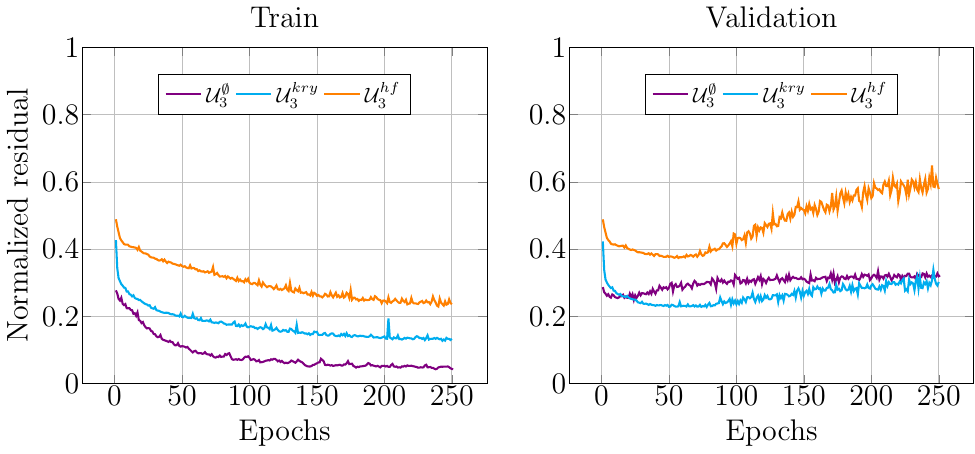}
    \caption{\small \small Convergence of train and test relative errors of the U-Net $\mathcal{U}_3^\star$ when trained on different datasets, for the physical constants $k_\Omega = 10^{-3}$, $\sigma_\Omega=10^{-3}$.}
    \label{fig:traintest}
\end{figure}

The performance of the proposed neural preconditioner is evaluated by comparing it with two well-established preconditioning strategies: the Incomplete LU (ILU) factorization and the Algebraic Multigrid (AMG) method \cite{saad2003iterative}. The ILU preconditioner approximates the LU factorization of a matrix by allowing a controlled level of fill-in, which provides a balance between sparsity and approximation accuracy, making it effective for moderately ill-conditioned systems. AMG is a hierarchical method that constructs a sequence of progressively coarser spaces to efficiently capture both low- and high-frequency error components, thereby accelerating convergence for a wide range of linear systems. Both ILU and AMG preconditioners have been implemented using the PETSc library (\url{https://petsc.org/release/})\cite{balay2023petsc}, which offers efficient and scalable linear algebra routines for large-scale computational problems, together with the PETSc python wrapper \texttt{petsc4py}(\url{https://pypi.org/project/petsc4py})\cite{dalcin2011parallel}.

{\color{black} All these preconditioners are tested on a flexible GMRES algorithm with right preconditioning, designated as $\text{FGMRES}(k)$,  restarted every $k=20$ steps. Iterations are stopped when the relative residual satisfies $||r||/||r_0|| \leq 10^{-6}$. In all tables, we show the average number of iterations required by the algorithm to converge (denoted as \textbf{Mean Iter.}), for an ensemble of linear systems corresponding to $10^2$ unseen graphs $\Lambda$, ranging from $O(1)$ to $O(10^2)$ branches.}



\subsection{Effect of the data-augmentation strategy}
We analyze the performance of the proposed preconditioners under three augmentation strategies: (i) no augmentation ($\mathcal{U}_3^\emptyset$), (ii) Krylov subspace augmentation ($\mathcal{U}_3^{kry}$) and (iii) random high-frequency vector augmentation ($\mathcal{U}_3^{hf}$). {\color{black} To explore the robustness of the approach, the tests are repeated for two values of the physical parameters. The results are reported in Table~\ref{tab:precond_comparison}.}

{\color{black} We discuss here the results for the most challenging case, $k_\Omega=\sigma_\Omega = 10^{-3}$.} The unaugmented U-Net preconditioner ($\mathcal{U}_3^\emptyset$) is less effective, requiring 92.41 iterations on average, indicating its limited ability to capture the full spectrum of the solution space. Incorporating Krylov subspace augmentation ($\mathcal{U}_3^{kry}$) reduces iterations significantly (24.54), as this strategy enhances the preconditioner's ability to handle higher-order components of the operator. Random vector augmentation ($\mathcal{U}_3^{hf}$) achieves the lowest mean iteration count (17.82), demonstrating superior performance over other versions of the U-net preconditioner. 

\begin{table}[h!]
    \centering
    \setlength{\tabcolsep}{5pt}
    \renewcommand{\arraystretch}{1}
    \begin{tabular}{lcccc|ccc}
    \hline 
        \textbf{Preconditioner} & \textbf{Type} & \multicolumn{3}{c}{$k_\Omega=\sigma_\Omega = 10^{-3}$} & \multicolumn{3}{c}{$k_\Omega=\sigma_\Omega = 10^{-2}$} \rule{0pt}{16pt} \vspace{5pt}\\        
        \hline
         &  & \textbf{Mean Iter.} & $\Delta^+$ & $\Delta^-$ & \textbf{Mean Iter.}& $\Delta^+$ & $\Delta^-$ \rule{0pt}{12pt}\\        
        \hline
        None & \_ & 147.72 & +83.28 & -52.72 & 187.04 & +121.96 & -70.04 \rule{0pt}{12pt}\\
        $\mathcal{U}_3^\emptyset$ & No Augmentation & 92.41 & +17.59 & -16.41 & 42.13 & +10.87 & -6.13 \rule{0pt}{12pt}\\
        $\mathcal{U}_3^{kry}$ & Krylov Subspace & 24.54 & +3.89 & -4.11 & 15.54 & +1.46 & -2.54 \rule{0pt}{12pt}\\
        $\mathcal{U}_3^{hf}$ & Random Vectors & 17.82 & +5.18 & -2.82 & 16.63 & +2.37 & -1.63 \rule{0pt}{12pt}\\ 
        \hline       
    \end{tabular}
    \caption{\small \small Mean FGMRES iterations for different preconditioners. The variability of the data corresponds to 100 test configurations.}
    \label{tab:precond_comparison}
\end{table}

\subsection{Impact of the 1D graph parameterization}
We also investigated to what extent including the distance function $\mathbf{d}_h^\mu$ as input enhances the effectiveness of the neural preconditioner $\mathcal{U}(\cdot, \mathbf{d}_h^\mu; \theta)$ for different instances of 1D domains. As shown in Table~\ref{tab:precond_comparison_distfun}, including the distance function $\mathbf{d}_h^\mu$ as an input significantly improves its performance, particularly in scenarios where the $\mu$-dependent component $M_{h,00}^\mu$ dominates. In fact, we recall that the structure of the linear system is defined as \(C_h^\mu := K_{h,00} + 2\pi\epsilon M_{h,00}^\mu\), where $K_{h,00}$ is the fixed component, and $M_{h,00}^\mu$ introduces the dependence of parameters. Without providing $\mathbf{d}_h^\mu$, the preconditioner mainly targets $K_{h,00}$, leading to lower performance on the whole system $C_h^\mu$. 

\begin{table}[h]
\centering
\setlength{\tabcolsep}{5pt}
\renewcommand{\arraystretch}{0.85}
\begin{tabular}{lccc|ccc}

\hline
\textbf{Preconditioner} & \multicolumn{3}{c}{$k_\Omega=\sigma_\Omega = 10^{-3}$} & \multicolumn{3}{c}{$k_\Omega=\sigma_\Omega = 10^{-2}$}  \rule{0pt}{16pt} \vspace{5pt}\\ 
\hline
 & \textbf{Mean Iter.} & $\Delta^+$ & $\Delta^-$ & \textbf{Mean Iter.} & $\Delta^+$ & $\Delta^-$ \rule{0pt}{12pt}\\ 
\hline
$\mathcal{U}_3^\emptyset(\cdot)$  & 108.65 & +30.35 & -29.65 & 41.50  & +8.50  & -5.50 \rule{0pt}{14pt}\\
$\mathcal{U}_3^\emptyset(\cdot, \mathbf{d}_h^\mu)$  & 92.41 & +17.59 & -16.41 & 42.13 & +10.87 & -6.13 \rule{0pt}{14pt}\\ 
$\mathcal{U}_3^{kry}(\cdot)$  & 44.91 & +8.09 & -7.91 & 20.69  & +2.31  & -3.69 \rule{0pt}{14pt}\\
$\mathcal{U}_3^{kry}(\cdot, \mathbf{d}_h^\mu)$  & 24.54 & +3.89 & -4.11  & 15.54 & +1.46 & -2.54 \rule{0pt}{14pt}\\ 
$\mathcal{U}_3^{hf}(\cdot)$  & 60.31 & +26.69 & -19.31 & 24.98  & +6.02  & -4.98
\rule{0pt}{14pt}\\
$\mathcal{U}_3^{hf}(\cdot, \mathbf{d}_h^\mu)$  & 17.82 & +5.18 & -2.82 & 16.63 & +2.37 & -1.63 \rule{0pt}{14pt} \vspace{5pt}\\
\hline
\end{tabular}
\caption{\small \small Mean FGMRES iterations with and without influence of the distance function.}
\label{tab:precond_comparison_distfun}
\end{table}

\subsection{Effect of pre- and post-smoothing}\label{sec:smoothing}
To improve the efficacy of the preconditioning method, similar to other studies \cite{Azulay2023S127,doi:10.1137/24M162861X}, we can employ techniques for pre- and post-smoothing. The smoothing procedure helps mitigate high-frequency error components, particularly for preconditioners that lack high-frequency training data. For a given residual $\mathbf{r}$ and initial guess $\mathbf{x}_0$, we consider a smoothing algorithm with Jacobi relaxation $\operatorname{S_J}(A, \mathbf{r}, \mathbf{x}_0, \text{maxit})$ \cite{trottenberg2000multigrid}.
The impact of pre-post smoothing is summarized in Table~\ref{tab:smoothing_effect}. The results show that preconditioning strategies with low-frequency training data ($\mathcal{U}_3^\emptyset$) benefit the most from smoothing, reducing iterations by 60\%.

{\color{black} This test sheds light on an important property of the data-augmented training strategy. Essentially, in Table~\ref{tab:smoothing_effect} we observe that pre- and post-smoothing has an effect comparable to that of data augmentation. This reveals that the data-augmentation strategy is equivalent to adding a pre- and post-smoother to the preconditioner, helping in learning how to process high frequencies and contrast the spectral bias. This property may turn out to be particularly effective in those cases where an efficient smoother is not known for the problem at hand. This conclusion is further supported by the fact that the high-frequency data augmentation strategy $\mathcal{U}_3^{hf}$ is less sensitive to smoother application.}

\begin{table}[h]
\centering
\setlength{\tabcolsep}{5pt}
\renewcommand{\arraystretch}{0.85}
\begin{tabular}{c c  c c c c | c c c c}
\hline
\textbf{Precond.} & \textbf{Relax} & \multicolumn{4}{c}{$k_\Omega=\sigma_\Omega = 10^{-3}$} & \multicolumn{4}{c}{$k_\Omega=\sigma_\Omega = 10^{-2}$} \rule{0pt}{16pt} \vspace{5pt} \\
\hline
 & \shortstack{\vspace{0.55cm}\;\\\;} & \shortstack{\textbf{Mean}\\\textbf{Iter.}} & $\Delta^+$ & $\Delta^-$ & \shortstack{\textbf{Rel.}\\\textbf{Gain.}} & \shortstack{\textbf{Mean}\\\textbf{Iter.}} & $\Delta^+$ & $\Delta^-$ & \shortstack{\textbf{Rel.}\\\textbf{Gain}} \rule{0pt}{14pt}\\ 
\hline
$\mathcal{U}_3^\emptyset$ & None & 92.41 & +17.59 & -16.41 &\multirow{2}{*}{0.64} & 42.13 & +10.87 & -6.13 &\multirow{2}{*}{0.47} \rule{0pt}{14pt}\\
$\mathcal{U}_3^\emptyset$ & Jacobi & 33.72 & +13.28 & -7.72 & & 22.47 & +3.53 & -4.47 & \rule{0pt}{14pt}\\ 

$\mathcal{U}_3^{kry}$ & None & 24.54 & +3.89 & -4.11 &\multirow{2}{*}{0.49} & 15.54 & +1.46 & -2.54 & \multirow{2}{*}{0.26} \rule{0pt}{14pt}\\ 
$\mathcal{U}_3^{kry}$ & Jacobi & 12.6 & +3.40 & -2.60 & & 11.46 & +0.54 & -1.46 &  \rule{0pt}{14pt}\\ 

$\mathcal{U}_3^{hf}$  & None & 17.83 & +5.17 & -2.83 &\multirow{2}{*}{0.26}& 16.63 & +2.37 & -1.63 & \multirow{2}{*}{0.22} \rule{0pt}{14pt} \rule{0pt}{14pt} \\
$\mathcal{U}_3^{hf}$  & Jacobi & 13.12 & +2.88 & -2.12 & &12.96 & +2.04 & -0.96  \rule{0pt}{14pt} \vspace{5pt}\\
\hline
\end{tabular}
\caption{\small \small Effect of pre- and post-smoothing on mean FGMRES iterations.}
\label{tab:smoothing_effect}
\end{table}

\subsection{Spectral analysis of training data and operator kernel}
As previously observed, the introduction of high-frequency vectors in the training set is beneficial, with the augmentation strategy based on random vectors slightly outperforming that using the {Kylov} subspace.  In this section, we seek to explain the observed behavior by contrasting the two training-set enrichment techniques with the operator's intrinsic kernel structure. In detail, to quantify the operator spectral content, we analyzed the null spaces of the problem matrices $K_{h,00}$, $2\pi\epsilon M_{00,h}^\mu$, and $C_h^\mu$ using Singular Value Decomposition (SVD).  Given a matrix \( A \in \mathbb{R}^{m \times n} \) with rank \( r \), the associated SVD is given by \(A = U \Sigma V^\top\), where \( U = [\mathbf{u}_1 \; | \; \mathbf{u}_2 \; | \; \ldots \; | \; \mathbf{u}_r] \) and \( V = [\mathbf{v}_1 \; | \; \mathbf{v}_2 \; | \; \ldots \; | \; \mathbf{v}_r] \) contain the left and right singular vectors, respectively. The right singular vectors $\mathbf{v}_i$ corresponding to zero singular values (\(\sigma = 0\)) span the null space, capturing kernel components that influence preconditioning performance. Using sparse SVD algorithms (\texttt{scipy.sparse.svds}), we extracted some of these
vectors and analyzed them in frequency space via the multidimensional discrete Fourier transform.
Let us denote by $\mathbf{V}$ the $d$-dimensional tensor corresponding to a right singular vector $\mathbf{v}$. Then we denote the DFT as $\mathcal{F}$, defined as follows,
\begin{equation*}
\mathcal{F}:  \mathbf{V} \in \mathbb{R}_{n_1} \times ... \times \mathbb{R}_{n_d}  \rightarrow  \mathcal{F}(\mathbf{V}) \in \mathbb{C}_{k_1}  \times ... \times \mathbb{C}_{k_d},\ \quad \ \mathcal{F}(\mathbf{V})|_{\mathbf{k}}=\sum_{\mathbf{n}=\mathbf{0}}^{\mathbf{N}-1} e^{-i 2 \pi \mathbf{k} \cdot(\mathbf{n} / N)} \mathbf{V}|_{\mathbf{n}},
\end{equation*}
where  ${\bf n}=(n_1,n_2, ..., n_d) \in \mathbb{N}^d$,  ${\bf k}=(k_1,k_2, ..., k_d)\in \mathbb{N}^d$ and $\mathbf{N}=(N,N, ..., N)\in \mathbb{N}^d$ are multi-integers and is $N$ the number of edge subdivision for the considered cubic structured mesh. 

\begin{figure}[h]
    \centering
    \includegraphics[width=\linewidth]{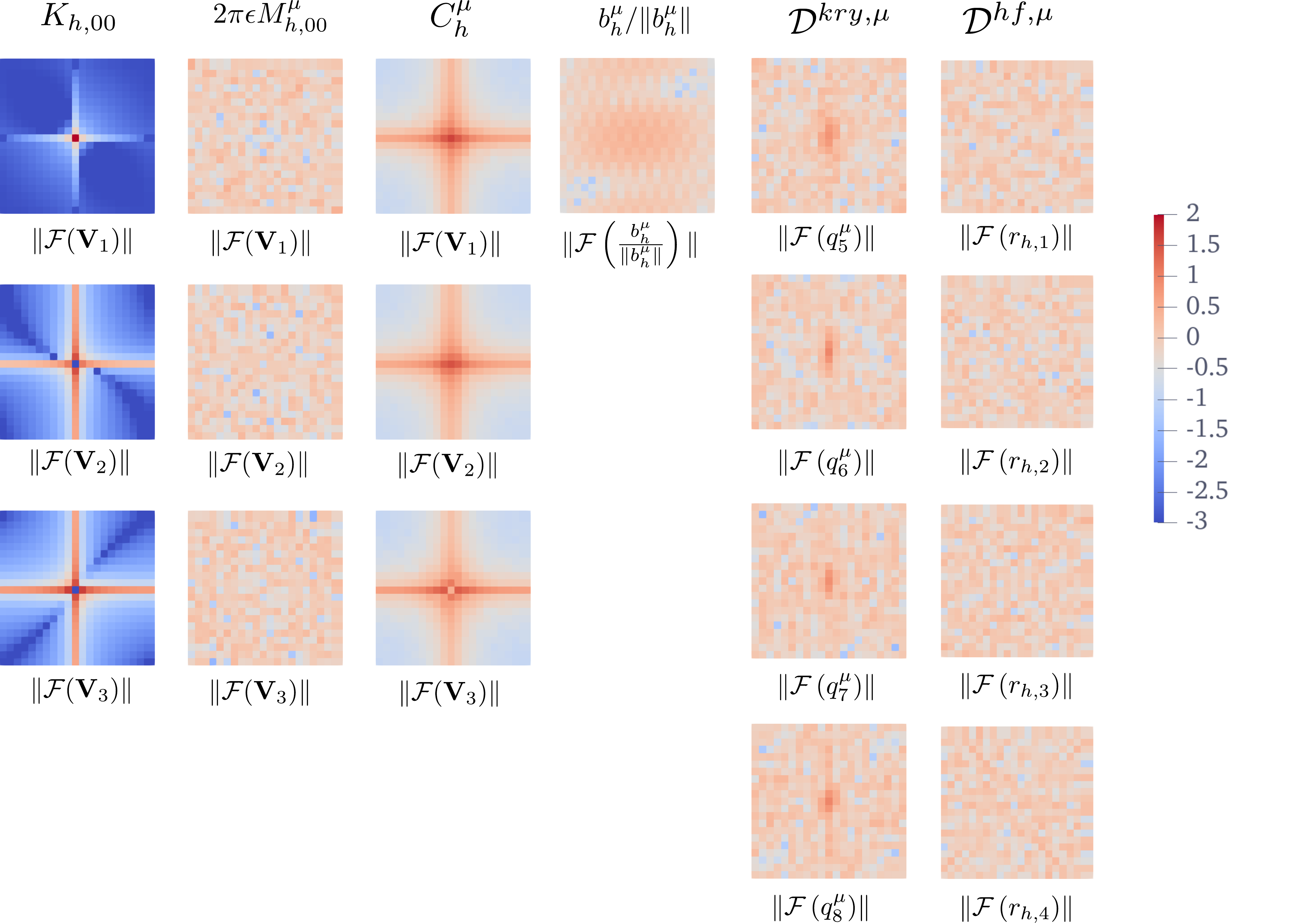}
    \caption{\small \small Logarithmic plot in the frequency domain $(k_1, k_2, k_3)$ of the $k_2k_3$-cross section for the spectrum magnitude $\|\mathcal{F}(\mathbf{V}_i)\|$; $k_1=0$ is considered and a value for $\mu$ is fixed given $k_\Omega=\sigma_\Omega =10^{-3}$. Frequency components increase with the distance from the center $(0, 0, 0)$. \textbf{Column I:} magnitude of first three right-singular vectors associated to $K_{h,00}$; \textbf{Column II:} magnitude of first three right-singular vectors associated to $2\pi\epsilon M_{h,00}^{\bar{\mu}}$; \textbf{Column III:} magnitude of first three right-singular vectors associated to $C_h^\mu=K_{h,00}+2\pi\epsilon M_{h,00}^{\bar{\mu}}$; \textbf{Column IV:} spectrum magnitude of the normalized forcing term $b_h^{\mu}$; \textbf{Column V:} spectrum magnitude of the orthonormal Krylov bases $\mathcal{D}^{kry, \mu}:=\{q_5^{\mu}, \ldots, q_8^{\mu}\}$ (note that $q_0^{\mu} \equiv b_h^{\bar{\mu}}/||b_h^{\bar{\mu}}||$); \textbf{Column VI:} spectrum magnitude of the random high frequency data augmentation set $\mathcal{D}^{hf, \mu}:=\{r_{h,1} \ldots r_{h,4}\}$.  
    }
    \label{fig:spectrum}
\end{figure}


In Figure~\ref{fig:spectrum} (columns I-III), two-dimensional slices of the spectrum magnitudes $\|\mathcal{F}(\mathbf{V}_i)\|$ associated with the kernel vectors $\{\mathbf{v}_i\}_{i=1,2,3}$ of matrices $K_{h,00}$, $ 2\pi\epsilon M_{h,00}^\mu$ and $C_h^\mu$ are shown.  For better visualization, the so-called \textit{ centered} spectrum is considered, where a suitable phase shift $e^{i \pi {\bf k}(N-1)}$ in the frequency space is applied to $\mathcal{F}(\mathbf{V})$; in this way, the point $\mathcal{F}(\mathbf{V})|_{(0,...,0)}$, is centered on the plot. We note that the frequency content in the spectra of matrices $K_{h,00}$ and $ 2\pi\epsilon M_{h,00}^\mu$ is remarkably different. As expected, the kernel of the metric term $ 2\pi\epsilon M_{h,00}^\mu$ introduces very noisy high-frequency components that need to be preconditioned.
Figure~\ref{fig:spectrum} (columns IV-VI) presents a comparison between the different data enhancement strategies, visualized in the frequency domain. Column IV shows the spectrum magnitude of the normalized forcing term. Column V illustrates the spectrum magnitude of the orthonormal Krylov basis. Finally, the last column presents the spectrum magnitude of the random high-frequency data augmentation set.
In Figure~\ref{fig:spectrum} we observe that Krylov subspace augmentation captures higher-order components by repeated application of the operator to a basis vector, producing structured high-frequency content (column V). 
In contrast, random vector augmentation introduces unstructured high-frequency content by randomly drawing components in an independent manner.
This results in a broader and more uniform spectral coverage (Fig.~\ref{fig:spectrum}, column VI), enabling the preconditioner to handle fine-scale variations in the solution, critical for mixed-dimensional problems like the 3D-1D coupled system.
From a comparison between the spectral content of different training sets and the spectrum of the kernel component of the metric term $2\pi\epsilon M_{h,00}^\mu$ (Fig.~\ref{fig:spectrum}, column II), one can conclude that the Krylov subspace augmentation introduces structured high-frequency components, while random vector augmentation spans a broader frequency range, closer to the true kernel spectrum of the operator. This broader frequency coverage can explain the better convergence behavior observed in numerical experiments, where random vector-based augmentation results are more effective than Krylov-based strategies in handling ill-conditioning and diverse parameter variations.

\subsection{Scalability with respect to the problem size}
Due to the intrinsic properties of the convolutional layers in the U-Net architecture, although trained on a low-resolution input tensor ${\bf X} \in \mathbb{R}^{21 \times 21 \times 21 \times 2}$, the proposed neural preconditioner $\mathcal{U}_3^{hf}$ can be applied to inputs of arbitrary dimensions, corresponding to higher resolutions of the problem. This adaptability makes it suitable for evaluating its scalability with respect to increasing problem size, quantified by the number of unknowns $N_h$ of the discrete problem. 

{\color{black} We test here the convergence rate of the preconditioned FGMRES for the solution of the 3D problem, for three different levels of mesh resolution $N$, namely $[21,41,81]$ points per side of the cubic domain $\Omega$ corresponding respectively to a number of degrees of freedom (d.o.f.), $N_h$, for a $\mathbb{P}^1$ FEM discretization equal to $[21^3;41^3;81^3]$. We measure the average number of iterations across 100 different test cases, all referring to 1D graphs that were not used during the training phase. Then, we analyze how the average iteration count scales with the number of d.o.f., $N_h$.

To assess the capacity of the neural preconditioner to generalize on the mesh resolution, we compare the neural preconditioner $\mathcal{U}_3^{hf}$ trained on two different resolutions, $N=21$ and $N=41$, respectively, and analyze how these two operators behave in the case of discretization based on $81$ points. The comparison also includes incomplete LU factorization (ILU) and Algebraic Multigrid (AMG) with 3 and 10 V-cycles, respectively.}\\

\begin{table}[htpb]
    \centering
    \setlength{\tabcolsep}{5pt}
    \renewcommand{\arraystretch}{1}
    \begin{tabular}{lll|cccccc} \hline
        \textbf{Precond.} & $N_h$ & $N$ & \textbf{Mean Iter.} & \textbf{Rate} ($N_h$) & \textbf{Time} & \textbf{Rate} ($N_h$) & \textbf{Time/Iter} & \textbf{Rate} ($N_h$) \\
        \hline 
         &9,261 & 21 & 147.72 & \_ & 365.33 & \_ & \_ & \_ \\ 
        \textbf{None}
         &68,921 & 41 & 431.91 & 0.53 & 7293.49 & 1.49 & \_ & \_ \\
         &531,441 & 81 & 1000* & \_ & \_ & \_ & \_ & \_ \\
        \hline
         
         &9,261 & 21 & 17.38 & \_ & 45.227 & \_ & 2.54 & \_ \\
        $\mathcal{U}_3^{hf}(21^3)$          
         &68,921 & 41 & 37.24 & 0.37 & 497.86 & 1.2 & 13.37 & 1.21 \\
         &531,441 & 81 & 100.73 & 0.49 & 9212.77 & 1.43 & 91.46 & 1.06 \\
        \hline

         &9,261 & 21 & 28.81 & \_ & 60.54 & \_ & 2.10 & \_ \\
        $\mathcal{U}_3^{hf}(41^3)$          
         &68,921 & 41 & 45.54 & 0.27 & 519.95 & 1.07 & 11.42 & 1.19 \\
         &531,441 & 81 & 119.27 & 0.47 & 11985.8 & 1.54 & 100.49 & 0.94 \\
        \hline

         &9,261 & 21 & 24.73 & \_ & 37.7 & \_ & 1.52 & \_ \\
        \textbf{ILU}          
         &68,921 & 41 & 46.2 & 0.31 & 486.04 & 1.27 & 10.52 & 1.04 \\
         &531,441 & 81 & 102.55 & 0.39 & 9422.78 & 1.45 & 91.88 & 0.94 \\
        \hline

         &9,261 & 21 & 9.56 & \_ & 62.35 & \_ & 6.52 & \_ \\
        \textbf{AMG}(3)          
         &68,921 & 41 & 17.13 & 0.29 & 696.14 & 1.20 & 40.64 & 1.10 \\
         &531,441 & 81 & 33.66 & 0.33 & 9631.04 & 1.29 & 286.13 & 1.05 \\
        \hline
         &9,261 & 21 & 3.64 & \_ & 55.67 & \_ & 15.29 & \_ \\
        \textbf{AMG}(10)          
         &68,921 & 41 & 4.01 & 0.05 & 329.32 & 1.02 & 82.12 & 1.19 \\
         &531,441 & 81 & 4.9 & 0.10 & 3921.46 & 1.07 & 800.30 & 0.90 \\
        \hline
    \end{tabular}
    \caption{\small \small Mean FGMRES iterations and execution time for different preconditioning strategies of the problem with $k_\Omega=\sigma_\Omega = 10^{-3}$. All times are in milliseconds. (*) Denotes that FGMRES did not converge in the maximum number of iterations.}
    \label{tab:merged_table_1}
\end{table}

    \bigskip
\begin{table}[htpb]
    \centering
    \setlength{\tabcolsep}{5pt}
    \renewcommand{\arraystretch}{1}
    \begin{tabular}{lll|cccccc}\hline
    \textbf{Precond.} & $N_h$ & $N$ & \textbf{Mean Iter.} & \textbf{Rate} ($N_h$) & \textbf{Time} & \textbf{Rate} ($N_h$) & \textbf{Time/Iter} & \textbf{Rate} ($N_h$) \\
    \hline
         & 9,261 & 21 & 187.94 & \_ & 216.17 & \_ & \_ & \_ \\
        \textbf{None}
         & 68,921 & 41 & 511.34 & 0.50 & 3,329.16 & 1.36 & \_ & \_ \\
         & 531,441 & 81 & 1000* & \_ & \_ & \_ & \_ & \_ \\
        \hline
         & 9,261 & 21 & 16.63 & \_ & 41.18 & \_ & 2.48 & \_ \\
        $\mathcal{U}_3^{hf}(21^3)$
         & 68,921 & 41 & 37.98 & 0.41 & 461.18 & 1.20 & 12.14 & 1.26 \\
         & 531,441 & 81 & 113.98 & 0.54 & 9,908.77 & 1.50 & 86.93 & 1.04 \\
        \hline
         & 9,261 & 21 & 24.27 & \_ & 55.73 & \_ & 2.30 & \_ \\
        $\mathcal{U}_3^{hf}(41^3)$
         & 68,921 & 41 & 41.39 & 0.27 & 458.34 & 1.05 & 11.07 & 1.28 \\
         & 531,441 & 81 & 104.99 & 0.46 & 9,586.00 & 1.49 & 91.30 & 0.97 \\
        \hline
         & 9,261 & 21 & 31.46 & \_ & 42.09 & \_ & 1.34 & \_ \\
        \textbf{ILU}
         & 68,921 & 41 & 62.18 & 0.34 & 548.30 & 1.28 & 8.82 & 1.06 \\
         & 531,441 & 81 & 147.30 & 0.42 & 12,123.32 & 1.52 & 82.30 & 0.91 \\
        \hline
         & 9,261 & 21 & 11.50 & \_ & 63.95 & \_ & 5.56 & \_ \\
        \textbf{AMG}(3)
         & 68,921 & 41 & 20.69 & 0.29 & 727.50 & 1.21 & 35.16 & 1.09 \\
         & 531,441 & 81 & 42.14 & 0.35 & 10,839.18 & 1.32 & 257.22 & 1.03 \\
        \hline
         & 9,261 & 21 & 3.92 & \_ & 53.25 & \_ & 13.58 & \_ \\
        \textbf{AMG}(10)
         & 68,921 & 41 & 4.00 & 0.01 & 404.91 & 1.01 & 101.23 & 0.999 \\
         & 531,441 & 81 & 4.00 & 0.00 & 3,450.91 & 1.05 & 862.73 & 0.95 \\
        \hline
    \end{tabular}

    \caption{\small \small Mean FGMRES iterations and execution time for different preconditioning strategies of the problem with $k_\Omega=\sigma_\Omega = 10^{-2}$. All times are in milliseconds. (*) Denotes that FGMRES did not converge in the maximum number of iterations.}
    \label{tab:merged_table_2}
\end{table}

\noindent{\color{black}
In Table~\ref{tab:merged_table_1} and~\ref{tab:merged_table_2}, we see that the scaling behavior of the neural preconditioner $\mathcal{U}_3^{hf}$ closely matches that of the ILU preconditioner. These methods show a dependence on $N_h$ of order $\simeq 0.4-0.5$.
The AMG(3) preconditioner shows a lower iteration count but features a similar scaling with respect to the size of the problem. None of these methods achieves optimal scaling as the AMG(10) preconditioner, the iteration rates of which are almost independent on $N_h$.

Finally, we note that the performance of the neural preconditioners $\mathcal{U}_3^{hf}(21^3)$ and $\mathcal{U}_3^{hf}(41^3)$ is almost equivalent, indicating that the size and resolution of the training data do not affect the overall performance of the neural operator. This means that it is possible to train the neural preconditioner on low-resolution data and deploy it on higher resolutions, with considerable savings on the computational cost of the offline training phase.

Previous analyses have focused solely on the number of FGMRES iterations as a performance metric. Although the neural preconditioners demonstrated sub-optimal results in terms of convergence rates, it is also essential to consider execution time, as the forward pass through a neural network is computationally lighter than, e.g., a single evaluation of the AMG algorithm. Although the implementation has not been optimized for performance (e.g., no batch-optimized algorithm or GPU acceleration techniques \cite{mittal2019survey}), we also present in Tables~\ref{tab:merged_table_1}-\ref{tab:merged_table_2} the results for execution time. Looking at the scaling with $N_h$ of the mean time required to solve the 3D problem, the advantage of AMG(10) is significantly reduced with respect to the neural preconditioner, ILU, and AMG(3), which provide aligned performances. This can be interpreted by observing that for FGMRES with sparse matrices, the overall computational cost is \( O(m \times N_h) \), where \( m \) is the number of iterations required for convergence. On the one hand, \( m \) is sensitive to the choice of the preconditioner; precisely as shown in Table~\ref{tab:merged_table_1}-\ref{tab:merged_table_2} column \textit{Mean Iter.} we have $m=(N_h)^{p_i}$ which being $p_i\simeq 0.5$ for $\mathcal{U}_3^{hf}$, $p_i\simeq 0.4$ for ILU, $p_i\simeq 0.4$ for AMG(3) and $p_i\simeq 0$ for AMG (10), this being the only case where $m$ is independent of $N_h$. On the other hand, the time per iteration scales linearly with $N_h$, that is, $T=c_iN_h$ for all methods, but the proportionality constant may be radically different, namely $c_i=1.723 10^{-4},\,1.883 10^{-4}$ for neural preconditioners trained on points $21^3$ and $41^3$, $c_i=1.721 10^{-4}$ for ILU, $c_i=5.383 10^{-4}$ for AMG(3) and $c_i= 14.98 10^{-4}$ for AMG(10), with neural and ILU preconditioners showing advantage over AMG. As a result, the following relation between the total time and the size of the problem holds, $\log T = c_i + (p_i+1) \log N_h$, showing that methods such as neural preconditioners and ILU provide a potential advantage for systems of small size.

We finally observe that the neural preconditioner struggles to address large systems, for example, the one arising from a discretization of the 3D domain with 101 points per side, corresponding to $N_h \simeq 10^6$, for which $\mathcal{U}_3^{hf}(21^3)$ requires on average 188.65 iterations to solve the 3D problem.  The behavior does not improve with increasing the dimension of the training set, as this is true for both $\mathcal{U}_3^{hf}(21^3)$ and $\mathcal{U}_3^{hf}(41^3)$. We hypothesize that this is due to the effect of high frequencies that become dominant components when the discrete solution belongs to a high-dimensional space. For this reason, we tested the preconditioner $\mathcal{U}_3^{hf}(21^3)$ with pre- and post-smoothing, as described in Section~\ref{sec:smoothing}, with very encouraging results given by an average number of iterations equal to $[13.12;27.60;62.32;91.55]$ for $N_h=[21^3;41^3;81^3;101^3]$, respectively.
This corresponds to the parameters $p_i=0.31$ for the dependence of the number of iterations on $N_h$ and $c_i=2.302 10^{-4}$ for the proportionality constant between time and iterations. This performance is in line with the ILU preconditioner, actually slightly better.}

{\color{black} 
\subsection{Coupling the neural preconditioner with ensembled calculations} 

In the context of parametric simulations of mixed-dimensional PDEs, solving multiple instances of the underlying linear system for varying configurations of the low-dimensional structure can become computationally prohibitive. To mitigate this, the so-called \textit{ensembled} framework has recently been developed to accelerate the solution of families of linear systems that share a common structure \cite{phipps2017ensemble,LIEGEOIS2020113188}. Ensembled methods consist of stacking multiple coefficient matrices and right-hand side vectors, each corresponding to a different parameter instance, and solving them simultaneously as a batched linear system. This paradigm is especially well-suited for settings such as uncertainty quantification or parametric sweeps, where many problems must be solved independently but share structural similarities. The main advantage of this approach lies in the optimized memory access.

We note that a key feature of the proposed neural preconditioner lies in its ability to generalize across the parameter space and that this property is well suited to ensemble methods, as the neural operator can process a group of inputs at the same time, each representing various parameter settings.



We have implemented a prototype of the \textit{ensembled} calculation using the algorithmic infrastructure developed for the neural preconditioner, based on the \texttt{PyTorch} library (\cite{paszke2019pytorch}) together with the CUDA kernel (\cite{luebke2008cuda}).   
In fact, we recall that the input of the neural preconditioner consists of tensors $\mathbf{X}_i \in \mathbb{R}^{c\times \mathbf{n} }$ with $c$ representing the number of channels and  ${\bf{n}}=n_1 \times n_2 \times ... \times n_d $ a multi-index encoding the physical dimension of the data (see Section~\ref{sec:unet}).  Given a set of tensors $\{\mathbf{X}_i\}_{i=1..N}$, it is possible to stack them along an additional \textit{batch dimension} to obtain a new tensor $\mathbf{X}_{e}=[{ \bf X}_1 |...|{\bf X}_N] \in \mathbb{R}^{N\times c \times {\bf n}}$ that encodes all input tensors to be fed to the neural preconditioner. 

We evaluated the speed-up factor achieved by implementing the neural preconditioner $\mathcal{U}_3^{hf}$ on $N \in [1,100]$ input tensors ${\bf X}_i \in \mathbb{R}^{2\times 21 \times 21 \times 21}$, which correspond 
to $N$ different 1D graphs in eq.~\ref{eq:decoupled_3d_problem}, i.e.
\begin{equation*}
    s=\frac{\sum_{i=1}^N t\left(\mathcal{U}_3^{hf}({\bf X}_i)\right)}{t\left(\mathcal{U}_3^{hf}({\bf X}_e)\right)}
\end{equation*}
with $t\left(\mathcal{U}_3^{hf}(\cdot)\right)$ representing the execution time of the neural preconditioner pass. The numerical tests are performed on an AMD EPYC 7301 CPU and a NVIDIA Tesla V100 GPU with 64 MB of dedicated RAM. 

From the results shown in Table~\ref{tab:batch-speedup_new} and Fig.~\ref{fig:speed-up}, we observe a speed-up around 4.5 when up to 100 input tensors are considered. Saturation occurs as the batched input approaches the memory limit of the hardware. The beneficial acceleration is associated with the contiguous memory structure, coupled with CUDA streams and execution via the PyTorch library. 
We prospect that the synergy between the ensembled approach and the neural preconditioner, along with the use of GPU based architectures, can enable significant gains in solver throughput, especially in multi-query scenarios. 
A comprehensive study encompassing computational architecture, algorithmic design, memory management, and preconditioner implementation is deferred to future research.}

\begin{figure}
    \centering
    \includegraphics[width=0.9\linewidth]{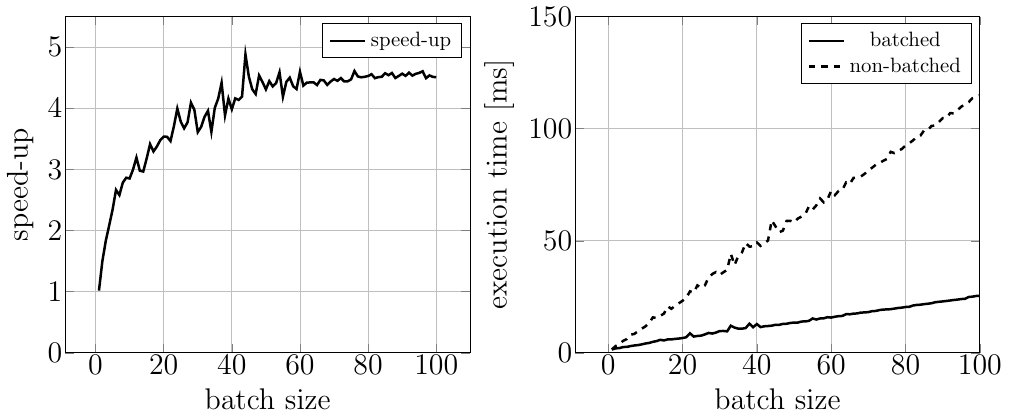}
    \caption{\small \textbf{Left:} plot of the speed-up associated to the batched execution of the neural preconditioning pass, using $\mathcal{U}_3^{hf}({\bf X}_e)$ with ${\bf X}_e \in \mathbb{R}^{N \times 2 \times 21 \times 21 \times 21} $, and batch size $N \in [1,100]$. \textbf{Right:} Plot of the execution time, expressed in millisecond, associated respectively to the serial (non-batched) and batched preconditioner pass of $\mathcal{U}_3^{hf}$. }
    \label{fig:speed-up}
\end{figure}

\begin{table}[htbp]
\centering
\begin{tabular}{l|cccccccccc}
\toprule
\textbf{Batch size} ($N$) & 1 & 5 & 10 & 15 & 20 & 30 & 50 & 70 & 90 & 100\\
\midrule
\textbf{time non-batched} [ms] & 1.5 & 6.1 & 11.8 & 17.5 & 23.2 & 34.7 & 58.2 & 81.3 & 104.7 & 114.9 \\
\midrule
\textbf{time batched} [ms] & 1.5 & 2.6 & 4.1 & 5.5 & 6.6 & 9.7 & 13.5 & 18.2 & 22.9 & 25.5 \\
\midrule
\textbf{Speed-up} & 1.0 & 2.3 & 2.9 & 3.5 & 3.7 & 3.6 & 4.3 & 4.5 & 4.6 & 4.5 \\
\bottomrule
\end{tabular}
\caption{\small Execution times [ms] and speed-up of the forward pass through the neural preconditioner as a function of the batch size.}
\label{tab:batch-speedup_new}
\end{table}




\section{Solution of the 3D-1D coupled problem}
\label{sec:coupled}

\begin{table}[b!]
    \centering
    \setlength{\tabcolsep}{5pt}
    \renewcommand{\arraystretch}{1}

\begin{tabular}{ll|cccc|cccc} \hline
    \textbf{Precond.} & $N_h$ & \multicolumn{4}{c|}{$k_\Omega=\sigma_\Omega = 10^{-3}$} & \multicolumn{4}{c}{$k_\Omega=\sigma_\Omega = 10^{-2}$} \rule{0pt}{16pt} \vspace{5pt} \\ 
    \hline
    & & \multirow{2}{*}{\shortstack{\textbf{Iter}}} & \multirow{2}{*}{\textbf{Rate}}  & \multirow{2}{*}{\textbf{Time}} & \multirow{2}{*}{\textbf{Rate}} & \multirow{2}{*}{\shortstack{\textbf{Iter}}} & \multirow{2}{*}{\textbf{Rate}}  & \multirow{2}{*}{\textbf{Time}} & \multirow{2}{*}{\textbf{Rate}}\\
    &&&&&&&&&\\ 
    \hline 
    
    \multirow{3}{*}{\small\shortstack{\small Non prec. \\ \small 3D-1D problem}}
    & 9,261 & 1973.55* &   & \multirow{3}{*}{n.r.}&  & 1768.65* & & \multirow{3}{*}{n.r.}&  \\ 
    & 68,921 & 2000* & &  & &2000* \\
    & 531441 & 2000* &  & & & 2000*\\
    \hline

    \multirow{3}{*}{\small\shortstack{\small AMG(10) \\ \small monolithic}}
    & 9,261 & 1181.48* & & \multirow{3}{*}{n.r.} & & 1220.83* & & \multirow{3}{*}{n.r.}  &\\ 
    & 68,921 & 1366.02* & &  & & 1423.44* \\
    & 531441 & 1862.71* &  & & & 1918.29*\\
    \hline

    \multirow{3}{*}{\small\shortstack{\small Haznics-AMG(10)\\ \small 3D-1D problem}}
    & 9,261 & 38.71 & \_  & 1.47& \_ & 35.36 & \_ & 1.49 & \_\\ 
    & 68,921 & 48.04 & 0.11 & 2.14 & 0.19 &41.39 & 0.08 & 2.26& 0.21\\
    & 531441 & 725.44* & 1.35 & 84.232 & 1.80& 497.57* & 1.22 & 64.95 & 1.64\\
    \hline

    \multirow{3}{*}{\small\shortstack{\small 1D: ILU \\ \small 3D: P-FGMRES \\ \small  P = none}}  
    & 9,261 & 153.82 & \_  &0.50 & \_ & 123.25 & \_ & 0.49 & \_\\ 
    & 68,921 & 417.82 & 0.50 &   8.51 &  1.42 & 316.88 & 0.47 & 4.48 & 1.10 \\
    & 531441 & 1564.84 & 0.65 &   260.41 &  1.32 & 916.19 & 0.52 & 152.88 & 1.73\\
    \hline

    \multirow{3}{*}{\small\shortstack{1D: ILU \\ 3D: P-FGMRES \\  P = $\mathcal{U}_3^{hf}(21^3)$}}
    & 9,261 & 60.74 & \_ & 0.47 & \_ & 22.28 & \_ & 0.27 & \_ \\  
    &68,921 & 123.30 & 0.35 &  3.12 &  0.94 & 54.33 & 0.44 & 1.54 & 0.87\\
    &531441 & 438.69 & 0.62 &  93.49 & 1.66 & 157.69 & 0.52 & 32.98 & 1.50 \\
    \hline

    \multirow{3}{*}{\small\shortstack{1D: ILU \\ 3D: P-FGMRES \\  P = $\mathcal{U}_3^{hf}(41^3)$}}
    & 9,261 & 69.15 & \_ & 0.51 & \_  & 26.95 & \_ & 0.29 & \_\\  
    &68,921 & 104.80 & 0.21 & 2.68  & 0.83 & 36.55 & 0.15 & 1.06 & 0.64\\
    &531441 & 257.91 & 0.44 &  54.62  & 1.48 & 59.99 & 0.23 & 12.78 & 1.22\\
    \hline

    \multirow{3}{*}{\small\shortstack{1D: ILU \\ 3D: P-FGMRES \\  P = AMG(10)}}
    &9,261 & 50.19 & \_   & 0.72 & \_ & 16.85 & \_ & 0.35 & \_\\ 
    &68,921 & 52.57 & 0.02   & 5.40 & 1.00 & 17.84 & 0.002 & 1.95 & 0.61  \\
    &531441 & 56.13 & 0.03  & 49.17 &  1.08 & 18.85 & 0.02 & 17.49 & 0.81 \\   
    \hline

    \multirow{3}{*}{\small\shortstack{1D: ILU \\ 3D: P-FGMRES \\  P = AMG(3)}}
    &9,261 & 51.53 & \_   & 0.63 & \_ & 18.23 & \_ & 0.33 & \_\\ 
    &68,921 & 56.61 & 0.05   & 4.64  & 0.99 & 23.00 & 0.12 & 2.01 & 0.90\\
    &531441 & 68.64 & 0.08 &  45.58  & 1.12 & 37.1 & 0.23 & 24.75 & 1.23\\   
    \hline

    \multirow{3}{*}{\small\shortstack{1D: ILU \\ 3D: P-FGMRES \\  P = ILU}}
    & 9,261 & 52.14 & \_ &  0.4 & \_  & 18.72 & \_ & 0.26 & \_\\ 
    & 68,921 & 58.08 & 0.05   & 1.97 & 0.79 & 26.52 & 0.17 & 1.05 & 0.70 \\
    & 531441 & 76.33 & 0.13   & 20.32 &  1.14 & 47.83 & 0.29 & 13.25 & 1.24\\   
    \hline
\end{tabular}
    
    \caption{\small \small Mean FGMRES iterations and execution times with associated rates for the preconditioning strategies considered, in different mesh configurations. Times are reported in seconds. Rates concerning iteration counts and execution times are computed with respect to $N_h$. (*) Denotes that the maximum number of iterations set to 2000 was exceeded. (n.r.) means that the data was not recorded.\\\\}
    \label{tab:iteration_and_time_FULL}
\end{table}

We now focus on the fully coupled 3D-1D problem, assessing how well the suggested block-preconditioning strategy \eqref{eq:right-prec full problem}-\eqref{eq:twosteps} performs when applied to the linear system \eqref{eq:mainsys}. We use a right-preconditioned FGMRES solver with a restart parameter $k=20$ and a stopping criterion $\|r\|/\|r_0\| \leq 1\cdot 10^{-15}$. This iterative solver is preconditioned by an approximation $\tilde{Q}_h^\mu$ of the block preconditioner defined in \eqref{eq:twosteps}. More precisely, the block preconditioner $\tilde{Q}_h^\mu$ is defined by an approximate solution to the 1D problem, obtained with ILU factorization, returning $\tilde{\zz}_{h,1}^\mu$ (a direct solver has also been tested, leading to no remarkable differences); then, 3 steps of the preconditioned FGMRES algorithm (precisely P-FGMRES with various preconditioners P) return $\tilde{\zz}_{h,0}^\mu$, an approximate solution of the 3D system. 

The performance of this approach for the 3D-1D coupled problem \eqref{eq:mainsys} with a comparison of monolithic solvers and different choices of preconditioners for the 3D problem can be found in Table~\ref{tab:iteration_and_time_FULL}. The latter presents a comparative analysis of the different strategies employed to solve the coupled 3D-1D system using the FGMRES iterative solver, with a focus on the impact of preconditioning. 
\newpage
As seen in the table, the results clearly highlight the importance of preconditioning for the efficient resolution of mixed-dimensional PDEs. Indeed, the first key observation is that the absence of a preconditioner results in a solver failing to converge within the budget of 2000 iterations. This emphasizes the severity of the ill-conditioning that characterizes the 3D-1D coupled problem and establishes the necessity of a robust preconditioning strategy.
The block-diagonal preconditioner $\tilde{Q}_h^\mu$ emerges as a good option, underlining the effectiveness of a decoupled block-structured approach that aligns with the intrinsic structure of the mixed-dimensional system. Interestingly, the full monolithic preconditioning strategy based on algebraic multigrid with 10 smoothing steps (denoted AMG(10) monolithic) fails to deliver competitive performance. This suggests that generic monolithic approaches are suboptimal for problems with strong 3D-1D coupling, as they are not designed to exploit the specific features of the system.
In contrast, the Haznics preconditioner, which involves a combination of AMG in the entire 3D domain and Schwarz smoother on the interface between the 3D and 1D domain \cite{10.1145/3625561}, shows the best performance in terms of iterations,  validating its role as the current state-of-the-art method for mixed-dimensional problems. A degradation of the mean iteration count can be observed for the highest value of $N_h$, especially due to lack of convergence when very complex graph geometries are considered.  

The table also reports the results for strategies that combine the block-diagonal preconditioner $\tilde{Q}_h^\mu$ with various solvers applied to its diagonal blocks. In particular, the first block, corresponding to the 3D sub-problem, is preconditioned with ILU, AMG (3), AMG (10) or different instances of the neural preconditioner, namely $\mathcal{U}^{\text{hf}}_3(21^3)$ and $\mathcal{U}^{\text{hf}}_3(41^3)$.
Although the iteration counts vary distinctively between the different strategies, we notice that their overall computational times are remarkably similar. This is a result of a trade-off between the cost per iteration and the number of iterations required. For instance, classical preconditioners like AMG generally result in lower iteration counts but at a higher per-iteration cost. Neural preconditioners, on the other hand, often yield higher iteration counts but with significantly lower cost per iteration, because of their efficient matrix-free representation. This balance ensures that all these strategies achieve similar total run-times, with performance comparable to or superior to monolithic alternatives.

These findings support the potential of the neural preconditioners as a viable and efficient strategy for solving complex mixed-dimensional problems of small-medium scale. The flexibility in trading iteration count for speed suggests that these learned preconditioners can serve as lightweight generalizable components in the broader preconditioning pipeline.

\section{Conclusions and Future Perspectives}

{\color{black}
In this work, we have developed a novel unsupervised neural network-based preconditioning strategy tailored to efficiently solve mixed-dimensional PDEs, specifically targeting the computational challenges arising from the coupling between 3D domains and embedded 1D structures. Using operator learning through convolutional neural networks, our approach successfully generalizes across various topological configurations of the 1D graphs and adapts robustly to different mesh resolutions without requiring retraining.

The neural preconditioner was extensively validated against classical preconditioning methods, including incomplete LU factorization and algebraic multigrid methods. Numerical experiments demonstrated that the proposed neural network preconditioner significantly accelerates convergence of iterative solvers, such as the Flexible GMRES, effectively handling both low- and high-frequency error components. Among the data augmentation strategies investigated, enriching the training set with random high-frequency vectors emerged as the most effective, 
enabling the neural network to overcome the inherent spectral bias common to most machine learning-based approaches.

Furthermore, our analysis underscored the critical role of including explicit parametric information through the distance function associated with the 1D graph, which improved the effectiveness of the preconditioner in varying geometric configurations. We also explored the complementary effects of pre- and post-smoothing strategies, illustrating their capacity to enhance preconditioning performance, particularly when the training set lacks high-frequency data.

However, there are limitations and areas for future improvement. A significant current limitation is that the proposed approach is applicable only to tensor-product domains, which restricts its use in realistic applications. To address this, we are actively exploring the use of graph- and mesh-informed neural networks \cite{gnns, FrancoMINN}, as suitable alternatives to convolutional layers, with the aim of extending the applicability of this method to more general computational domains. Additionally, developing a solid theoretical framework for selecting data augmentation subsets and, more broadly, for learning-based preconditioning methods is crucial to consolidate this approach and broaden the range of application areas. Although the scalability of the neural preconditioner shows promise, more work is required to enhance its performance in large-scale problems. 

The promising results obtained from this study not only establish neural preconditioning as a viable and efficient alternative to traditional preconditioners for mixed-dimensional PDEs, but also lay a robust foundation for further exploration into broader classes of coupled multiphysics problems. 
}

\section*{Acknowledgments}
{\small The authors thank Prof. Kent-Andre Mardal and Dr. Miroslav Kuchta of the Simula Research Laboratory for useful discussions.
PZ acknowledges the support of the MUR PRIN 2022 grant No. 2022WKWZA8 \emph{Immersed methods for multiscale and multiphysics problems} (IMMEDIATE) part of the Next Generation EU program Mission 4, Comp. 2, CUP D53D23006010006. NF has received funding under the project \textit{Reduced Order Modeling and Deep Learning for the real-time approximation of PDEs (DREAM)}, grant no. FIS00003154, funded by the Italian Science Fund (FIS) and by Ministero dell'Università e della Ricerca (MUR). The present research is part of the activities of the \emph{Dipartimento di Eccellenza} 2023-2027, Department of Mathematics, Politecnico di Milano. All authors are members of the Gruppo Nazionale per il Calcolo Scientifico (GNCS), Istituto Nazionale di Alta Matematica (INdAM).}

\bibliographystyle{plain} 
\bibliography{references}





\appendix 
\section{Auxiliary results}
\label{sec:appendix}
\renewcommand{\thesection}{\Alph{section}}

\begin{lemmas}
    \label{lemma:inversion}
    Let $(V,\|\cdot\|_{V})$ and $(W,\|\cdot\|_{W})$ be two Banach spaces. Let $\invertible(V,W)\subset\linearmaps(V,W)$ be the set of invertible bounded linear operators from $V$ to $W$, equipped with the operator norm
    $$\|A\|_{\linearmaps(V,W)}:=\sup_{v\in V\setminus\{0\}}\displaystyle\frac{\|Av\|_{W}}{\|v\|_{V}}.$$
    Then,
    \begin{itemize}
        \item [1.] $\invertible(V,W)$ is open in the operator norm topology;
        \item [2.] for every $A\in\invertible(V,W)$ there exists a unique inverse $A^{-1}\in\linearmaps(W,V).$  Furthermore, $A^{-1}\in\invertible(W,V);$
        \item [3.] the inversion map $\mathscr{I}:A\mapsto A^{-1}$ is $\invertible(V,W)\mapsto \invertible(W,V)$ continuous.
    \end{itemize}
\end{lemmas}

\begin{proof} We mention that a stronger result holds, as $\mathscr{I}$ is in fact holomorphic: see, e.g., the proof in \cite[Proposition 4.8]{adcock2022sparse}. Nonetheless, we can easily prove the statements in the Lemma as follows.
\;\\
    \begin{itemize}
        \item [1.] Given $A\in\invertible(V,W)$, let $B\in\linearmaps(V,W)$ be such that $\|A-B\|_{\linearmaps(V,W)}<1/\|A^{-1}\|_{\linearmaps(W,V)}.$
        We notice that the operator $C=A^{-1}(A-B)=\identity_{V}-A^{-1}B$ is an endomorphism from $V$ onto itself. Furthermore,
        $$\|C\|_{\linearmaps(V,V)}\le \|A^{-1}\|_{\linearmaps(W,V)}\cdot\|A-B\|_{\linearmaps(V,W)}<1.$$
        It follows that, for every $k$, the $k$th power of $C$, defined via composition, satisfies $\|C^{k}\|_{\linearmaps(V,V)}\le \|C\|_{\linearmaps(V,V)}^{k}<1.$
        As a consequence, the Nuemann series
        $\sum_{k=0}^{+\infty}C^{k}$
        is shown to be absolutely convergent in $\linearmaps(V,V)$. Indeed,
        $$\sum_{k=0}^{+\infty}\|C^{k}\|_{\linearmaps(V,V)}\le
        \sum_{k=0}^{+\infty}\|C\|_{\linearmaps(V,V)}^{k} = \frac{1}{1-\|C\|_{\linearmaps(V,V)}}.$$
        Since $\linearmaps(V,V)$ is a Banach space, this ensures that $\sum_{k=0}^{+\infty}C^{k}$ converges to some element of $\linearmaps(V,V)$. In particular, by construction of the Neumann series,
        $\sum_{k=0}^{+\infty}C^{k}=(\identity_{V}-C)^{-1}=(A^{-1}B)^{-1}.$ Let
        \begin{equation}
        \label{eq:binverse}
        B^{\dagger}:=(A^{-1}B)^{-1}A^{-1}.
        \end{equation}
        Clearly, $B^{\dagger}\in\linearmaps(W,V)$ by composition. Furthermore, we have
        $A^{-1}BB^{\dagger}=A^{-1}.$
        Since $A^{-1}$ is invertible, and thus of full rank, this shows that $BB^{\dagger}=\identity_{W}$. Similarly, composing \eqref{eq:binverse} with $B$ yields
        $$B^{\dagger}B=(A^{-1}B)(A^{-1}B)=\identity_{V}.$$
        It follows that $B\in\invertible(V,W)$ and $B^{\dagger}=B^{-1}.$\\

        \item[2.] This follows immediately from the fact that invertible operators have full rank. To see this, let $A\in \invertible(V,W)$ and assume $B_{1},B_{2}\in\linearmaps(W,V)$ are two -possibly different- inverses of $A$. For all $w\in W$ we have
        $$B_{1}w = B_{2}A(B_{1}w) = B_{2}(AB_1w)=B_2w$$
        since $B_1w\in V$ and $AB_1w\in W$. Then $B_1\equiv B_2=:A^{-1}.$ The fact that $A^{-1}\in\invertible(W,V)$ is obvious.\\

        \item[3.] Let $A_{n}\to A$ in $\invertible(V,W)$. Define the error $E_{n}=A-A_{n}.$ There exists some $n_{0}$ such that
        $\|E_{n}\|_{\linearmaps(V,W)}<1/\|A\|_{\linearmaps(V,W)}$
        for all $n\ge n_0$. As for our proof of (1), we notice that this ensures the convergence of the following Neumann series
        $$\sum_{k=0}^{+\infty}(A^{-1}E_{n})^{k}=(\identity_V-A^{-1}E_{n})^{-1}.$$
        whenever $n\ge n_0$. Since
        $$
        \|A^{-1}-A_{n}^{-1}\|_{\linearmaps(W,V)}\le\|\identity_{V}-A_{n}^{-1}A\|_{\linearmaps(V,V)}\cdot\|A^{-1}\|_{\linearmaps(W,V)},$$
        it suffices to prove that $\|\identity_{V}-A_{n}^{-1}A\|_{\linearmaps(V,V)}\to0$ as $A_{n}\to A$. To this end, we note that for all $n\ge n_0$
        \begin{align*}
         \|\identity_{V}-A_n^{-1}A\|_{\linearmaps(V,V)}
        &=
        \|\identity_{V}-(A-E_{n})^{-1}A\|_{\linearmaps(V,V)}
        \\ &=\|\identity_{V}-(\identity_{V}-A^{-1}E_{n})^{-1}A^{-1}A\|_{\linearmaps(V,V)}\\&=
        \|\identity_{V}-(\identity_{V}-A^{-1}E_{n})^{-1}\|_{\linearmaps(V,V)}.
        \end{align*}
        \noindent In particular, $\|\identity_{V}-A_n^{-1}A\|_{\linearmaps(V,V)}=\left\|\sum_{k=1}^{+\infty}(A^{-1}E_{n})^{k}\right\|_{\linearmaps(V,V)}$, and thus
        $$
            \|\identity_{V}-A_n^{-1}A\|_{\linearmaps(V,V)} \le\|A^{-1}\|_{\linearmaps(W,V)}\sum_{k=1}^{+\infty}\|E_{n}\|^{k}_{\linearmaps(V,W)}=\frac{\|A^{-1}\|_{\linearmaps(W,V)}\|E_{n}\|_{\linearmaps(V,W)}}{1-\|E_{n}\|_{\linearmaps(V,W)}},
        $$
        which is vanishing for $n\to +\infty$, as $E_{n}\to0.$
    \end{itemize}
\end{proof}


\end{document}